\documentclass[letterpaper, oneside]{amsart}

\usepackage{layout}	

\usepackage[
]{geometry}

\usepackage[parfill]{parskip}   
\usepackage{amssymb}
\usepackage{amsthm}
\usepackage{amsmath}
\usepackage{cite}
\usepackage[all]{xy}
\usepackage{enumerate}
\usepackage{verbatim}
\usepackage{mathtools}
\usepackage{hyperref}
\usepackage{colonequals}
\usepackage{mathrsfs}
\usepackage{color}

\theoremstyle{plane}
\newtheorem{thm}{Theorem}[section]
\newtheorem{lemma}[thm]{Lemma}

\newtheorem{prop}[thm]{Proposition}
\newtheorem{cor}[thm]{Corollary}

\theoremstyle{definition}
\newtheorem{defn}[thm]{Definition}
\newtheorem*{rmk}{Remark}
\newtheorem*{ack}{Acknowledgement}
\newtheorem*{claim}{Claim}

\newtheorem{example}[thm]{Example}

\newcommand{\A}{\mathbb{A}}
\newcommand{\F}{\mathbb{F}}
\newcommand{\Z}{\mathbb{Z}}
\newcommand{\N}{\mathbb{N}}
\newcommand{\Q}{\mathbb{Q}}
\let\P\relax
\newcommand{\P}{\mathbb{P}}
\newcommand{\C}{\mathbb{C}}
\newcommand{\bY}{\mathbb{Y}}

\let\k\relax
\newcommand{\k}{\mathbf{k}}

\let\O\relax
\newcommand{\O}{\mathcal{O}}
\newcommand{\X}{\mathcal{X}}
\newcommand{\Y}{\mathcal{Y}}
\newcommand{\B}{\mathcal{B}}
\let\H\relax
\newcommand{\H}{\mathcal{H}}

\newcommand{\cF}{\mathcal{F}}
\newcommand{\cC}{\mathcal{C}}
\newcommand{\U}{\mathcal{U}}

\newcommand{\g}{\mathfrak{g}}
\let\b\relax
\newcommand{\b}{\mathfrak{b}}
\newcommand{\n}{\mathfrak{n}}

\newcommand{\f}{\mathfrak{f}}
\newcommand{\fS}{\mathfrak{S}}

\newcommand{\s}{\mathbf{s}}
\let\u\relax
\newcommand{\u}{\mathbf{u}}

\let\l\relax
\newcommand{\l}{\ell}
\newcommand{\act}[1]{{}^{#1}}
\newcommand{\cl}[1]{\overline{#1}}

\DeclareMathOperator{\ad}{\textup{ad}}
\DeclareMathOperator{\supp}{\textup{supp}}
\DeclareMathOperator{\Spec}{\textup{Spec}}

\DeclareMathOperator{\Pic}{\textup{Pic}}
\DeclareMathOperator{\ch}{\textup{char}}
\DeclareMathOperator{\del}{\partial}

\title{Homology Class of a Deligne-Lusztig variety and its analogues}
\author{Dongkwan Kim}
\address{Department of Mathematics\\
  Massachusetts Institute of Technology\\
  Cambridge, MA 02139-4307\\
  U.S.A.}
\email{sylvaner@math.mit.edu}
\date{\today}							

\begin{document}
\begin{abstract} In this paper we consider Deligne-Lusztig varieties and their analogues when the Frobenius endomorphism is replaced with conjugation by an element in a group, especially a regular semisimple or regular unipotent one. We calculate their classes in the Chow group of the flag variety in terms of Schubert classes. Also we give some sufficient criteria when different elements in the Weyl group result in the same class. 
\end{abstract}

\maketitle

\renewcommand\contentsname{}
\tableofcontents

\section{Introduction}
Let $\k$ be an algebraically closed field and $G$ be a connected reductive group defined over $\k$. If $G$ is in addition defined over $\F_q$ for $q$ power of $\ch \k$, we denote by $F$ the geometric Frobenius morphism corresponding to $\F_q$. Let $W$ be the Weyl group of $G$ and $\B$ be the flag variety of $G$.

In this article we are interested in a Deligne-Lusztig variety corresponding to $w \in W$, denoted $X(w)$ \cite{dl} and its analogues when we replace $F$ with conjugation by $g\in G$, denoted $\Y_{w,g}$. We are mainly focused on their homology classes in the Chow group of $\B$, denoted $A_*(\B)$. \cite{hansen} showed that $[\cl{X(w)}]_{w\in W}$ form a basis of $A_*(\B)_\Q$, and we strengthen this result so that it is true in $A_*(\B)_{\Z[1/|G^F|_{p'}]}$ where $|G^F|_{p'}$ is the largest factor of the number of $F$-fixed points in $G$ prime to $\ch \k$. Also we describe $[\cl{X(w)}] \in A_*(\B)$ explicitly using Schubert classes.

Moreover, we analyze properties of $\Y_{w, g}$ for $g\in G$. Note that this variety is related to the construction of character sheaves of Lusztig. (Indeed, $\Y_{w,g}$ is the same as the fiber of $Y_w \rightarrow G$ defined in \cite{lu:char1}.) Especially we are focused on the case when $g=\s$ regular semisimple or $g=\u$ regular unipotent. (The former case is studied in \cite{lu:reflection}, and the latter case is studied in \cite{kawanaka}, \cite{lu:weyltounip}, and \cite{lu:homogeneity}.) In this paper we calculate the class of such varieties in terms of Schubert classes and give some conditions when different elements in $W$ gives the same class in $A_*(\B)$.

\begin{ack}
I thank George Lusztig for suggesting this topic and giving thoughtful comments.
\end{ack}

\section{Notations and definitions}
In this paper $\k$ is an algebraically closed field of characteristic $p$ which can be zero. We denote by $G$ a connected reductive group over $\k$. We fix a Borel subgroup $B$ and a maximal torus $T$ contained in $B$. We denote by $B^-$ the opposite Borel subgroup of $B$. Define $U, U^-$ to be the unipotent radical of $B, B^-$, respectively. If $G$ is defined over a finite field we usually require that $B$ and $T$ are rational.

Let $W$ be the Weyl group of $G$, canonically identified with $N(T)/T$ where $N(T) \subset G$ is the normalizer of $T$. We denote by $S \subset W$ the set of simple reflections and by $\l: W \rightarrow \N$ the standard length function. We usually write $w_0 \in W$ to be the longest element with respect to this length function on $W$. Note that the following notion is well-defined.

\begin{defn} For $w\in W$, $\supp(w) \subset S$, called \emph{the support of $w$}, is the set of simple reflections which appear in some/any reduced expression of $w$.
\end{defn}

Define $\B$ to be the flag variety of $G$, which parametrizes all Borel subgroups of $G$, identified with $G/B$. Note that there exists a Bruhat decomposition
$$ \B \times \B = \bigsqcup_{w \in W} \O_w
$$
where each $\O_w$ is a diagonal $G$-orbit indexed by elements in the Weyl group $W$. For $(B', B'') \in \O_w$ we say that $B'$ and $B''$ are in relative position $w$, denoted $B' \sim_w B''$. We define a Schubert variety as follows.
\begin{defn} For $w \in W$, \emph{the Schubert variety corresponding to $w$}, denoted $C_w$, is a locally closed subvariety of $\B$ defined by 
$$C_w \colonequals \{ B' \in G \mid B \sim_w B'\}.$$
\end{defn}

If $p\neq 0$ and a variety $X$ over $\k$ is defined over $\F_q\subset \k$, then we denote by $F$ a geometric Frobenious morphism corresponding to $\F_q$. Define $X^F$ to be the set of fixed points in $X$ by $F$. Also denote $F(x)$ by $\act{F}x$ for simplicity. For a finite set $A$, define $|A|$ to be the cardinal of $X$. Thus $|X^F|$ is the number of fixed points by $F$ on $X$ if finite. If $G$ is a group over $\F_q$ we let $L: G \rightarrow G : g \mapsto g^{-1}\act{F}g$ be the Lang map. $L$ is surjective if $G$ is a linear algebraic group over $\F_q$. \cite[Theorem 10.1]{steinberg}

We recall the definition of Deligne-Lusztig varieties.

\begin{defn} \cite[1.4]{dl} Suppose $G$ is defined over $\F_q$ and $F$ is the geometric Frobenius corresponding to $\F_q$. For $w \in W$, we define \emph{the Deligne-Lusztig variety associated to $w$} by
\begin{displaymath}
X(w) \colonequals \{ B' \in \B \mid B' \sim_w \act{F}B'\}.
\end{displaymath}
\end{defn}

For $g \in G$, we may also define an analogous variety where $F$ is replaced by $\ad(g)$. As before we denote $\act{g}x \colonequals \ad(g)(x)$ for an object $x$ to which $\ad(g)$ can be applied.

\begin{defn}\label{def:analogue} For $g \in G$ and $w \in W$, we let
\begin{displaymath}
\Y_{w,g} \colonequals \{ B' \in \B \mid B' \sim_w \act{g}B'\}.
\end{displaymath}
\end{defn}
We are interested in some special cases when $g$ is either regular semisimple or regular unipotent. We usually denote by $\s\in G$ a regular semisimple element and by $\u \in G$ a regular unipotent one.

For a variety $X$ over $\k$, define $A_*(X)$ to be \emph{the Chow group of $X$}. If $X$ is smooth over $\k$, there is a natural ring structure on $A_*(X)$ which we also call \emph{the Chow ring of $X$} in this case. For a ring $R$, we denote $A_*(X)_R \colonequals A_*(X) \otimes_{\Z} R$.

For a subvariety we assume that it is always closed. It might be confusing as Schubert varieties and Deligne-Lusztig varieties are in general not closed. But it will be apparent based on the context.
Also for a variety $X$ and a subvariety $Y, Z \subset X$, $Y \cap Z$ always indicates the set-theoretic intersection with reduced scheme structure otherwise specified.

\section{Class of a Deligne-Lusztig variety}
In this section $\ch \k = p \neq 0$ and assume that $G$ is defined over $\F_q \subset \k$ for $q$ some power of $p$. Let $F$ be the geometric Frobenius corresponding to $\F_q$. $X(w)$ is smooth of pure dimension $\ell(w)$ for any $w \in W$ by argument after Definition 1.4. in \cite{dl}. 

We fix a rational Borel subgroup $B$ and a rational maximal torus $T \subset B$ of $G$. We naturally identify $\B \simeq G/B$ and $W \simeq N(T)/T$. 

\subsection{The number of components of $X(w)$}
We first consider the number of (irreducible) components of $X(w)$. To that end we need some lemmas as follows.

\begin{lemma}\label{dl:irred} $X(w)$ is irreducible if and only if $\bigcup_{n \in \N} \supp(\act{F^n}w) = S$.
\end{lemma}
\begin{proof} \cite[Example 3.10(d)]{lu:chevalley}.
\end{proof}

\begin{lemma}\label{dl:ind} Let $I= \bigcup_{n \in \N} \supp(\act{F^n}w) \subset S$. Suppose $P_I$ is a rational parabolic subgroup corresponding to $I$ which contains $B$. We have a Levi decomposition $P_I = L_I U_I$ where $L_I$ is the (rational) Levi subgroup that contains $T$. Then we have an isomorphism of $G^F\times F^m$-varieties
$$ G^F/U_I^F \times_{L_I^F} X_{L_I}(w) \simeq X(w)
$$
where $X_{L_I}(w)$ is a Deligne-Lusztig variety for $L_I$ if we regard $w$ as an element of $W_I$, the parabolic subgroup of $W$ corresponding to $I$, and where $m \in \mathbb{N}$ satisfies $\act{F^m}w=w$.
\end{lemma}
\begin{proof} \cite[Proposition 2.3.8]{dmr:coh}.
\end{proof}
From above we have the following proposition.

\begin{prop}\label{dl:comp} The number of components of $X(w)$ is $|G^F/P_I^F|$, where $P_I$ is as defined in Lemma \ref{dl:ind}. Moreover, $G^F$ acts transitively on the set of components of $X(w)$.
\end{prop}
\begin{proof} Since $X_{L_I}(w)$ is irreducible by Lemma \ref{dl:irred}, the components of $X(w)$ is in bijection with the points of $G^F/U_I^FL_I^F = G^F/P_I^F$ as $G^F$-sets by Lemma \ref{dl:ind}, from which the result follows.
\end{proof}

\subsection{Linear independence of $[\cl{X(w)}]$ in the Chow group of a flag variety}
We are interested in the class $[\cl{X(w)}] \in A_*(\B)$. First we show that $\{[\cl{X(w)}]\}_{w \in W}$ form a basis of $A_*(\B)_\Q$.

\begin{prop}\label{dl:indep} $\{[\cl{X(w)}]\}_{w \in W}$ is a linear basis of $A_*(\B)_\Q$.
\end{prop}
\begin{proof} We slightly revise the proof of \cite[Proposition 3.19]{hansen}. Since the cardinal of $\{[\cl{X(w)}]\}_{w \in W}$ is the same as the dimension of $A_*(\B)_\Q$, it suffices to show that they are linearly independent. Recall that we fixed a rational Borel subgroup $B$ and a rational maximal torus $T \subset B$ of $G$. It is known that the natural projection $\varphi : G/T \rightarrow G/B$ gives an isomorphism (up to degree shift) of Chow groups
$$\varphi^* : A_*(G/B) \rightarrow A_*(G/T).
$$
Thus it suffices to show that $\{\varphi^*[\cl{X(w)}]\}_{w\in W} = \{[\varphi^{-1}(\cl{X(w)})]\}_{w\in W}$ are linearly independent in $A_*(G/T)_\Q$. To that end we define the "$\Q$-dual", denoted $Y(w)$ for $w\in W$, as the following. Recall that $L: G \rightarrow G$ is the Lang map.
$$Y(w):= \{ gT \in G/T \mid L(g) \in B^- w B\}$$
Here $B^-$ is the opposite Borel subgroup with respect to $B$. Note that $Y(w)$ is well-defined and of dimension $\ell(w_0) - \ell(w)+ \dim U$ where $w_0 \in W$ is the longest element in $W$ and $U \subset B$ is the unipotent radical of $B$.

Let $\tilde{\pi} : G \rightarrow G/T$ be the obvious projection map. For $w, w' \in W$, we have
\begin{align*}
\cl{\varphi^{-1}(X(w))} \cap \cl{Y(w')} &= \tilde{\pi} (L^{-1}(\cl{BwB} \cap \cl{B^- w' B}))
\end{align*}
since $\tilde{\pi}$ and $L$ are fiber bundles. Now if $\ell(w) =\ell(w'),$ $\cl{BwB} \cap \cl{B^- w' B}$ is nonempty if and only if $w=w'$ and thus so is $\cl{\varphi^{-1}(X(w))} \cap \cl{Y(w)}$. Thus $[\cl{\varphi^{-1}(X(w))}]$ are linearly independent for $w\in W$ of a fixed length, which means that $[\cl{\varphi^{-1}(X(w))}]$ are linearly independent for all $w\in W$. This is what we want to prove.
\end{proof}

Indeed, it is even true if we replace $\Q$ by $\Z[1/(q|G^F|)]$. Later we refine this result so that we only need to invert "prime to $p$" part of $|G^F|$.
\begin{prop} \label{dl:basis}Let $N=q|G^F|$. Then $[\cl{X(w)}]_{w\in W}$ form a linear basis of $A_*(\B)_{\Z[1/N]}$.
\end{prop}

\begin{proof}  Consider the quotient morphism $\psi: \B \rightarrow G^F\backslash \B$ which is finite of degree $|G^F|$. We construct $\psi^* : A_*(G^F\backslash \B)_{\Z[1/N]} \rightarrow A_*(\B)_{\Z[1/N]}$ by mimicking \cite[Example 1.7.6]{fulton}. Indeed, for any irreducible subvariety $D \subset \B$, let
$$I_D = \{ g \in G^F \mid g|_D=id_D\}, \qquad e_D = \frac{|I_D|}{\deg_i(D/\psi(D))}$$
where $\deg_i(D/\psi(D))$ is the degree of inseparability of $\k(D)$ over $\k(\psi(D))$. Note that $\deg_i(D/\psi(D))$ is a power of $p = \textup{char } \k$ thus is invertible in $\Z[1/N]$. For a variety $X$, define $Z_*(X)$ to be a free abelian group generated by irreducible subvarieties of $X$. For an irreducible subvariety $D' \subset G^F\backslash \B$ we define
$$\psi^*: Z_*(G^F\backslash \B)_{\Z[1/N]} \rightarrow Z_*(\B)_{\Z[1/N]}: [D'] \mapsto \sum_{D\subset \psi^{-1}(D')} e_D[D]$$
where the sum is over all the irreducible components $D$ of $\psi^{-1}(D')$. It is clear that it descends to a morphism $A_*(G^F\backslash \B)_{\Z[1/N]} \rightarrow A_*( \B)_{\Z[1/N]}$, which we again denote by $\psi^*$. It is known \cite[Example 1.7.6]{fulton} that it is an isomorphism after base change to $\Q$, and the endomorphism $\psi_*\circ \psi^*$ (where $\psi_* :  A_*(\B)_{\Z[1/N]} \rightarrow A_*(G^F\backslash \B)_{\Z[1/N]}$ is a proper push-forward) on $A_*(G^F\backslash \B)_{\Z[1/N]}$ is multiplication by $|G^F|$, hence indeed an automorphism.

We claim that $\psi^*$ is surjective. Indeed, for any irreducible subvariety $D_0 \subset G^F\backslash \B$, let $D'$ be the image of $D_0$ under $\psi$. We have
$$\psi^*([D']) = \sum_{D \subset \psi^{-1}(D')}e_D[D]$$
where the sum is over all the irreducible components $D$ of $\psi^{-1}(D')$. However, as $G^F \subset G$ acts trivially on the Chow group and $G^F$ acts transitively on the set of irreducible components of $\psi^{-1}(D')$, we have $\psi^*([D']) = ne_{D_0}[D_0]$ for $n$ the number of irreducible components of $\psi^{-1}(D')$. Thus we have
$$\frac{1}{n e_{D_0} \deg(D_0/D')}\psi^*\psi_*([D_0]) = \frac{1}{n e_{D_0}}\psi^*([D']) = [D_0].$$
Note that $n e_{D_0} \deg(D_0/D')$ is invertible in $\Z[1/N]$. Indeed, clearly $n$ divides $|G^F|$ and $e_{D_0}$ is a unit as we observed above. For $\deg(D_0/D')$, the separable degree divides $|G^F|$ and the inseparable degree is a power of $p$ which is invertible in $\Z[1/N]$.

Since $\psi_*\circ \psi^*$ is an automorphism and $\psi^*$ is surjective, it follows that $\psi^*$ and $\psi_*$ are isomorphisms. In particular,  $A_*( G^F\backslash\B)_{\Z[1/N]}$ is free as $A_*(\B)_{\Z[1/N]}$ is free. Now we need a following lemma which will be proved later. (Note that $G^F\backslash \cl{X(w)}=\cl{G^F\backslash X(w)}$ is irreducible by Proposition \ref{dl:comp}.)
\begin{lemma} \label{dl:aux}$\{[G^F\backslash \cl{X(w)}]\}_{w \in W}$ is a basis of $A_*(G^F\backslash\B)_{\Z[1/N]}$.
\end{lemma}
From Proposition \ref{dl:comp}, we have
$$\psi^*([G^F\backslash \cl{X(w)}]) = \frac{|G^F|}{|P_I^F|}e_D[D] = e_D [\cl{X(w)}]$$
where $D$ is some irreducible component of $\cl{X(w)}$. Again this is independent of the choice of $D$. Since $e_D$ is invertible in $\Z[1/N]$ and $\psi^*$ is an isomorphism, it follows from Lemma  \ref{dl:aux} that $\{[\cl{X(w)}]\}_{w\in W}$ is a basis of $A_*(\B)_{\Z[1/N]}$ as desired.
\end{proof}

\begin{proof}[Proof of Lemma \ref{dl:aux}] 
If suffices to show that $\{[G^F\backslash \cl{X(w)}]\}_{w\in W}$ generates $A_*(G^F\backslash \B)_{\Z[1/N]}$. We label elements in $w \in W$ by $c: W \rightarrow \{1, \cdots, |W|\}$ such that $w\geq v \Rightarrow c(w) \geq c(v)$. (In other words, we linearize the Bruhat order of $W$.) For $1\leq n \leq |W|$, define $V_n = \bigcup_{c(w) \leq n} G^F\backslash X(w)$ which is a subvariety of $G^F\backslash \B$ not necessarily irreducible. 

We prove that $A_*(V_n)_{\Z[1/N]}$ is generated by $\{[G^F\backslash \cl{X(w)}]\}_{c(w) \leq n}$ by induction on $n$, which implies the claim when $n = |W|$. For $n=1$ it is trivial. In general, let $w \in W$, $n=c(w)$ and recall the following exact sequence.
$$A_*(V_{n-1})_{\Z[1/N]}\rightarrow A_*(V_n)_{\Z[1/N]} \rightarrow A_*(G^F\backslash X(w))_{\Z[1/N]} \rightarrow 0$$
We show that $A_k(G^F\backslash X(w))_{\Z[1/N]}=\Z[1/N]$ if $k = \ell(w)$ and 0 otherwise. Indeed, define $\pi: \tilde{X}(\dot{w}) \rightarrow X(w)$ to be the $T^{wF}$-torsor on $X(w)$ as in \cite[1.8]{dl} corresponding to a fixed representative $\dot{w} \in N(T)$ of $w\in W$. It descends to a morphism $\bar{\pi} : G^F \backslash \tilde{X}(\dot{w}) \rightarrow G^F \backslash X(w)$ which is also a quotient morphism by $T^{wF}$. (This is not a $T^{wF}$-torsor in general.) Using the same argument as in Proposition \ref{dl:basis} we construct $\bar{\pi}^* : A_*(G^F \backslash X(w))_{\Z[1/N']} \rightarrow A_*(G^F \backslash \tilde{X}(\dot{w}))_{\Z[1/N']}$ where $N' = q|T^{wF}|$. But $T^{wF}$ is isomorphic to the group of Frobenius-fixed elements in a rational torus of type $w$ of $G$\cite[Definition 3.24]{dm:book}, thus $N'$ divides $N$. As a result $\bar{\pi}^*$ is well-defined over $\Z[1/N]$.

Consider
$$A_*(G^F \backslash X(w))_{\Z[1/N]} \xrightarrow{\bar{\pi}^*} A_*(G^F \backslash \tilde{X}(\dot{w}))_{\Z[1/N]} \xrightarrow{\bar{\pi}_*} A_*(G^F \backslash X(w))_{\Z[1/N]}$$
where the composition is multiplication by $|T^{wF}|$, hence an automorphism. Now we claim the following.
\begin{claim}
$A_k(G^F \backslash \tilde{X}(\dot{w}))_{\Z[1/N]} = \Z[1/N]$ if $k = \ell(w)$, 0 otherwise.
\end{claim}
If this is true, then it implies that  $A_k(G^F \backslash X(w))_{\Z[1/N]}=0$ for $k \neq \ell(w)$. Also clearly $A_{\ell(w)}(G^F \backslash X(w)) = \Z[1/N]$ since $G^F \backslash X(w)$ is irreducible of dimension $\ell(w)$. Therefore for $k \neq \ell(w)$ the above map $A_k(V_{n-1})_{\Z[1/N]}\rightarrow A_k(V_n)_{\Z[1/N]}$ is surjective. For $k=\ell(w)$ it is also clear that $A_{\ell(w)}(V_n)_{\Z[1/N]}$ is generated by its $\ell(w)$-dimensional irreducible components, which are $[G^F\backslash \cl{X(v)}]$ for $\{ v\in W \mid c(v) \leq n, \ell(v) = \ell(w)\}$. By induction, the result follows.

Now we prove the claim. Indeed, consider the following diagram.
$$\xymatrixcolsep{5pc}\xymatrix{\tilde{X}(\dot{w}) \simeq L^{-1}(\dot{w}U)/U\cap \act{\dot{w}}U \ar[d]^g& L^{-1}(\dot{w}U) \ar[l]_-{f'}\ar[d]^{g'}\\
G^F\backslash\tilde{X}(\dot{w}) & \dot{w}U \ar[l]_f}
$$
Here $f$ is a quotient by $U \cap \act{\dot{w}}U$, $g$ is a quotient by $G^F$, $f'$ is a $U \cap \act{\dot{w}}U$-torsor, and $g'$ is a $G^F$-torsor. The action of $U\cap \act{\dot{w}}U$ on $L^{-1}(\dot{w}U)$ is given by right multiplication, and that on $\dot{w}U$ is given by
$$x \in U\cap \act{\dot{w}}U, \quad y \in \dot{w}U, \quad\quad x\cdot y \colonequals x^{-1}y\act{F}x.$$
$G^F$ acts on $L^{-1}(\dot{w}U)$ and $\tilde{X}(\dot{w})$ by left multiplication. For detailed description of these morphisms one may refer to \cite[Corollary 1.12]{dl}. 

As before we have morphisms of Chow groups \cite[Example 1.7.6]{fulton}
$$A_*(\dot{w}U)_{\Z[1/N]} \xrightarrow{g'^*} (A_*(L^{-1}(\dot{w}U))_{\Z[1/N]})^{G^F} \xrightarrow{g'_*} A_*(\dot{w}U)_{\Z[1/N]}.$$
We claim that $g'^*: A_*(\dot{w}U)_{\Z[1/N]} \rightarrow (A_*(L^{-1}(\dot{w}U))_{\Z[1/N]})^{G^F}$ is surjective. Indeed, for any $E \in (Z_*(L^{-1}(\dot{w}U))_{\Z[1/N]})^{G^F}$, $E = \frac{1}{|G^F|}\sum_{g \in G^F}g_*E$. (For a variety $X$, $Z_*(X)$ is a free abelian group generated by irreducible subvarieties of $X$ as in the proof of Proposition \ref{dl:basis}.) But for any irreducible subvariety $V \subset L^{-1}(\dot{w}U)$, $\frac{1}{|G^F|}\sum_{g \in G^F}g_*[V]$ is in the image of $g'^*$ by similar argument to the proof of Proposition \ref{dl:basis}. Thus $g'^*: Z_*(\dot{w}U)_{\Z[1/N]} \rightarrow (Z_*(L^{-1}(\dot{w}U))_{\Z[1/N]})^{G^F}$ is surjective. But since taking $G^F$-invariant is an exact functor as $|G^F|$ is invertible in $\Z[1/N]$, $(Z_*(L^{-1}(\dot{w}U))_{\Z[1/N]})^{G^F} \rightarrow (A_*(L^{-1}(\dot{w}U))_{\Z[1/N]})^{G^F}$ is surjective. It follows that $g'^*: A_*(\dot{w}U)_{\Z[1/N]} \rightarrow (A_*(L^{-1}(\dot{w}U))_{\Z[1/N]})^{G^F}$ is also surjective.

Since $g'_*\circ g'^*$ is a multiplication by $|G^F|$, hence an automorphism, the above argument implies that $g'_*$ and $g'^*$ are both isomorphisms. Since $A_k(\dot{w}U) = \Z[1/N]$ if $k=\ell(w_0)$ and 0 otherwise, so is $(A_k(L^{-1}(\dot{w}U))_{\Z[1/N]})^{G^F}$. Also similar argument shows that we have an isomorphism $g^*: A_*(G^F\backslash \tilde{X}(\dot{w}))_{\Z[1/N]} \rightarrow (A_*(\tilde{X}(\dot{w}))_{\Z[1/N]})^{G^F}$. Thus it suffices to show that $(A_*(\tilde{X}(\dot{w}))_{\Z[1/N]})^{G^F}$ is isomorphic to $(A_*(L^{-1}(\dot{w}U))_{\Z[1/N]})^{G^F}$ with degree shifted by $\ell(w_0)-\ell(w)$. Indeed, the flat pull-back $f'^* : A_*(\tilde{X}(\dot{w}))_{\Z[1/N]} \rightarrow A_*(L^{-1}(\dot{w}U))_{\Z[1/N]}$ is an isomorphism. \cite[Corollary A.18]{hansen} Since $f'^*$ is $G^F$-equivariant, it induces an isomorphism $f'^* : (A_*(\tilde{X}(\dot{w}))_{\Z[1/N]})^{G^F} \rightarrow (A_*(L^{-1}(\dot{w}U))_{\Z[1/N]})^{G^F}$ by functoriality. As the fiber of $f'^*$ is $U \cap \act{\dot{w}}U$ of dimension $\ell(w_0)-\ell(w)$, we get the result.
\end{proof}

\subsection{$[\cl{X(w)}]$ in terms of Schubert classes}

In this section we calculate the class of a Deligne-Lusztig variety in terms of the Schubert classes, i.e. $[\cl{C_w}]$ for $w\in W$. Let $(id, F): \B \rightarrow \B\times \B$ be a closed embedding of $\B$ into $\B \times \B$ which identifies $\B$ with the graph of the Frobenius morphism $\Gamma_F$. Note that $\Gamma_F \cap \O_w$ is isomorphic to $X(w)$ by this morphism.

We have $\cl{X(w)} = \bigsqcup_{v \leq w} X(v)$ where $v \leq w$ is with respect to the (strong) Bruhat order on $W$. Indeed, let $\pi: G \rightarrow G/B$. Then $\pi^{-1}(X(w)) = L^{-1}(\O_w)$. Since $\pi$ and $L$ are fiber bundles, we have $\pi^{-1}(\cl{X(w)}) = L^{-1}(\cl{\O_w})$. Now we use the fact that $\O_w = \bigsqcup_{v \leq w} \O_v$. (This is the definition of the Bruhat order on $W$; it is equivalent to other definitions by e.g. \cite[Proposition 6]{chevalley}.) But this implies that $\Gamma_F \cap \cl{\O_w} = \cl{\Gamma_F \cap \O_w}$.

If the intersection $\Gamma_F \cap \cl{\O_w}$ is generically transversal, then $[\cl{X(w)}]$ is equal to $$(id, F)^*[\cl{\O_w}]$$ where $(id, F)^*: A_*(\B \times \B) \rightarrow A_*(\B)$ denotes the Gysin homomorphism. (ref. \cite[Chapter 6]{fulton}) Indeed, this is the case.

\begin{lemma} The intersection of $\Gamma_F$ with $\cl{\O_w}$ is generically transversal. More precisely, the intersection of $\Gamma_F$ with $\O_w$ is (nonempty and) transversal.
\end{lemma}

\begin{proof} The proof is the same as that of \cite[Lemma 9.11]{dl}.
\end{proof}

$A_*(\B \times \B)$ has a basis $\{[\cl{C_{w'}} \times \cl{C_{w''}}]\}_{w', w'' \in W}$. If we express $[\cl{\O_w}]$ with respect to this basis, we obtain the following result.

\begin{prop}\label{orbitdecomp} In $A_*(\B \times \B)$ we have
$$[\cl{\O_w}] = \sum_{\substack{u, v \in W \\ uwv^{-1}=w_0\\\ell(u)+\ell(v)=\ell(w_0)-\ell(w)}}[\cl{C_{w_0u}} \times \cl{C_{w_0v}}].$$
\end{prop}
\begin{proof} Note that for $w_1, w_2, w_3, w_4 \in W$ such that $\ell(w_1)+\ell(w_2)+\ell(w_3)+\ell(w_4) = 2\ell(w_0)$, we have
\begin{align*}[\cl{C_{w_1}} \times \cl{C_{w_2}}]\cdot[\cl{C_{w_3}} \times \cl{C_{w_4}}]=
\begin{cases}
      1 & \text{if } w_1w_3^{-1} = w_2w_4^{-1} = w_0,\\
      0 & \text{otherwise.}
\end{cases}
\end{align*}
Since $\cl{\O_w}$ is of dimension $\ell(w_0) + \ell(w)$, thus it suffices to show that for $u, v \in W$ such that $\ell(u)+\ell(v)=\ell(w_0)-\ell(w)$ we have
\begin{align*}[\cl{\O_w}]\cdot[\cl{C_{u}} \times \cl{C_{v}}]=
\begin{cases}
      1 & \text{if } uwv^{-1}=w_0,\\
      0 & \text{otherwise.}
\end{cases}
\end{align*}
To that end, we consider the intersection of $\cl{\O_w}$ with $\cl{C_u} \times w_0\cl{C_v}$ for such $u, v \in W$. Indeed, $\cl{\O_w} \cap (\cl{C_u} \times w_0\cl{C_v})$ equals
\begin{align*}
 & \{ (B', B'') \in \B \mid \exists u'\leq u, v'\leq v, w'\leq w \text{ s.t. } B\sim_{u'} B', B' \sim_{w'} B'', B^- \sim_{v'} B'' \}
\\&=\{ (B', B'') \in \B \mid B\sim_{u} B', B' \sim_{w} B'', B^- \sim_{v} B'' \}
\end{align*}
by the length condition and the fact that $B\sim_{w_0} B^-$. This intersection is nonempty if and only if $uwv^{-1} = w_0$ and in such case it consists of a single point.

It remains to show that this intersection is transversal. Note that if the intersection $\cl{\O_w} \cap (\cl{C_u} \times w_0\cl{C_v})$ is nonempty then the point of intersection $(B', B'')$ actually lies in $\O_w \cap (C_u \times w_0C_v)$. Since $\O_w$ and $C_u \times w_0C_v$ are smooth, it suffices to show that the tangent spaces of these varieties span the whole tangent space of $\B \times \B$ at $(B', B'')$. First we identify the tangent space of $\B \times \B$ at $(B', B'')$ with $\g/\b' \times \g/\b''$ where $\g$ is the Lie algebra of $G$ and $\b', \b''$ are the Lie algebras of $B', B''$, respectively. Also we let $\b, \b^-$ be the Lie algebras of $B, B^-$, respectively. Then the tangent space of $\O_w$ at $(B', B'')$ is identified with
$$\{(\pi'(x), \pi''(x)) \in \g/\b' \times \g/\b'' \mid x \in \g \} $$
where $\pi' : \g \rightarrow \g/\b', \pi'' : \g \rightarrow \g/\b''$ are natural projections. Likewise, the tangent space of $C_u\times w_0 C_v$ at this point is identified with $(\b + \b'/\b') \times (\b^-+ \b''/\b'')$. 

Suppose an arbitrary $(x_1, x_2) \in \g \times \g$ is given. Then we may find $(y_1, y_2) \in \b \times \b^-$ such that $x_1-x_2  =y_1-y_2$ since $\b+\b^-=\g$. Let $y \colonequals x_1 - y_1 = x_2-y_2 \in \g$. Then $x_1 = y+y_1$ and $x_2 = y+y_2$, thus $\pi'(x_1) = \pi'(y) + \pi'(y_1)$ and $\pi''(x_2) = \pi''(y)+\pi''(y_2)$. Since $(x_1, x_2)$ is given arbitrary, the result follows.
\end{proof}

\begin{rmk} For $w=id$, it is the same as the class of the diagonal in $\B \times \B$. \cite[Lemma 3.1.1]{brion}
\end{rmk}

As a result, if we can calculate $(id, F)^*[\cl{C_{w'}} \times \cl{C_{w''}}]$ then it is easy to deduce $(id, F)^* [\cl{\O_w}]$ from the proposition above. Indeed it has a simple expression as follows.

\begin{prop} For $u,v \in W$,  $(id, F)^*[\cl{C_{u}} \times \cl{C_{\act{F}v}}] = [\cl{C_u}]\cdot[\cl{C_{v}}] q^{\ell(w_0) -\ell(v)}$ in $A_*(\B)$.
\end{prop}

\begin{proof} To that end, we consider the intersection product $\Gamma_F\cdot (\cl{C_u}\times w_0 \cl{C_{\act{F}v}})$, which is equal to $(id, F)^*[ \cl{C_u}\times \cl{C_{\act{F}v}}]$ after push-forward by the closed embedding $\Gamma_F\cap (\cl{C_u}\times w_0 \cl{C_{\act{F}v}}) \hookrightarrow \B$. (For the definition of intersection product, one may refer to \cite{fulton}.) Indeed, the set-theoretic intersection $\Gamma_F\cap (\cl{C_u}\times w_0 \cl{C_{\act{F}v}})$ is
\begin{align*}
&\{(B', \act{F}B') \in \B \times \B \mid \exists u' \leq u, v' \leq v \text{ s.t. } B \sim_{u'} B', B^- \sim_{\act{F}v'} \act{F}B'\}
\\&\simeq \{B' \in \B \mid \exists u' \leq u, v' \leq v \text{ s.t. } B \sim_{u'} B', B^- \sim_{v'} B'\}\simeq\cl{C_u} \cap w_0\cl{C_{v}}
\end{align*}
which is the Richardson variety corresponding to the pair $(u, v)$ in $W$. It is clearly irreducible. Therefore we have $\Gamma_F \cdot (\cl{C_u}\times w_0 \cl{C_{\act{F}v}}) = m[\Gamma_F\cap (\cl{C_u}\times w_0 \cl{C_{\act{F}v}})]$ where $m$ is the multiplicity of $\Gamma_F\cap (\cl{C_u}\times w_0 \cl{C_{\act{F}v}})$ in $\Gamma_F \cdot (\cl{C_u}\times w_0 \cl{C_{\act{F}v}})$. Also proper push-forward by $\Gamma_F\cap (\cl{C_u}\times w_0 \cl{C_{\act{F}v}}) \hookrightarrow \B$ maps it to $m[\cl{C_u}\cap w_0 \cl{C_{v}}] = m[\cl{C_u}]\cdot[\cl{C_{v}}].$ Thus it remains to prove that $m = q^{\ell(w_0)-\ell(v)}.$

To calculate $m$, first we choose an affine open subset $V = \Spec A \subset \B-\bigcup_{u'<u} C_{u'}-\bigcup_{v'<{v}}w_0C_{v'}$ which meets $\cl{C_u} \cap w_0\cl{C_{v}}  \subset \B$, so that $V\times \act{F}V \subset \B \times \B$ meets $\Gamma_F\cap (\cl{C_u}\times w_0 \cl{C_{\act{F}v}})$. Also, we denote by $C'_u,w_0C'_v, w_0C'_{\act{F}v}, \Gamma'_F$ the intersections of $\cl{C_u}, w_0\cl{C_{v}},w_0\cl{C_{\act{F}v}}, \Gamma_F$ with $V,V, \act{F}V, V\times \act{F}V$, respectively. Note that $C'_u, w_0C'_v,w_0C'_{\act{F}v}$ are subsets of $C_u, w_0C_v,w_0C_{\act{F}v}$, respectively, hence smooth. Now it suffices to calculate the intersection multiplicity of $C'_u\times w_0C'_{\act{F}v}$ with $\Gamma'_F$ where the set-theoretic intersection is isomorphic to $C'_u\cap w_0C'_{v}$.

We have the following cartesian diagram.
$$\xymatrix{Z\ar[d]\ar[r]&W\ar[d]\ar[r]&\ar[d]\Gamma'_F\\C'_u \times w_0 C'_{\act{F}v} \ar[r]&V\times w_0C'_{\act{F}v}\ar[r] &V\times \act{F}V}$$
Here $W = (V\times w_0C'_{\act{F}v}) \times_{V\times \act{F}V} \Gamma'_F$ and $Z = (C'_u \times w_0 C'_{\act{F}v})\times_{V\times \act{F}V} \Gamma'_F$. Indeed, we have 
\begin{align*}
Z_{red} &= (C'_u \times w_0 C'_{\act{F}v})\cap \Gamma'_F =\{(B', \act{F}B') \in V\times \act{F}V \mid B\sim_u B', B^-\sim_vB'\} \simeq C'_u\cap w_0C'_{v}, 
\\W_{red} &= \Gamma_F' \cap (V\times w_0C'_{\act{F}v}) =\{(B', \act{F}B') \in V\times \act{F}V \mid B' \in w_0C'_v\}\simeq  w_0 C'_v.
\end{align*}
where $W_{red}, Z_{red}$ denote $W, Z$ with reduced scheme structure, hence varieties. Also they are irreducible.

We use commutativity and associativity of intersection product \cite[Example 7.1.7, 7.1.8]{fulton}. Note that morphisms on the second row are regular embeddings since $C'_u, w_0C'_{\act{F}v}$ are smooth. Thus $m$ equals the product of intersection multiplicities of $Z_{red}$ on $(C'_u \times w_0 C'_{\act{F}v})\cdot W_{red}$ and $W_{red}$ on $(V\times w_0C'_{\act{F}v})\cdot \Gamma'_F$. First we claim that $W_{red}$ and $C'_u \times w_0 C'_{\act{F}v}$ intersects transversally. Indeed, $W_{red}$ and $C'_u \times w_0 C'_{\act{F}v}$ are smooth, thus it suffices to show that for any $(B', \act{F}B') \in W_{red} \cap  (C'_u \times w_0 C'_{\act{F}v}),$ the tangent spaces of the two span the whole tangent space of $V \times w_0C'_{\act{F}v}$ at $(B', \act{F}B')$.

The tangent space of $V\times w_0C'_{\act{F}v}$ at $(B', \act{F}B')$ is identified with $\g/\b' \times (\b^-+\act{F}\b')/\act{F}\b'$ where $\g, \b, \b^-, \b'$ are Lie algebras of $G, B, B^-, B'$, respectively. Similarly, the tangent space of $C'_u\times w_0C'_{\act{F}v}$ at this point is $(\b+\b')/\b' \times (\b^-+\act{F}\b')/\act{F}\b'$. The tangent space of $W_{red}$ at this point is given by
$$\{(\vec{v}, \act{F}\vec{v}) \in \g/\b' \times (\b^-+\act{F}\b')/\act{F}\b' \mid \vec{v} \in \b^-\}.$$
Thus we easily see that the tangent spaces of two subvarieties span the whole tangent space at $(B', \act{F}B')$. Since the choice of $(B', \act{F}B')$ was arbitrary, the transversality follows.

Thus we are done if the intersection multiplicity of $W_{red}$ on $(V\times w_0C'_{\act{F}v}) \cdot \Gamma'_F$ is $q^{\ell(w_0)-\ell(v)}$. To that end, we first identify $V\times \act{F}V \simeq \Spec A\otimes A$, and let $J \subset A$ be the ideal of $w_0C'_v \subset V$ such that $V \times w_0C'_{\act{F}v} \simeq \Spec A\otimes A/1\otimes J$. Also, $\Gamma'_F \simeq \Spec A\otimes A/\langle x^q\otimes 1-1\otimes x \mid x \in A\rangle$, thus isomorphic to $\Spec A$ under the isomorphism $A \xrightarrow{\simeq} A\otimes A/\langle x^q\otimes 1-1\otimes x \mid x \in A\rangle : a \mapsto a\otimes 1$. Now the scheme-theoretic intersection $W$ corresponds to 
$$A/J^q \simeq A\otimes A/(\langle x^q\otimes 1-1\otimes x \mid x \in A\rangle+J)$$
under $a \mapsto a\otimes 1$. Since $\Gamma'_F$ is smooth, thus Cohen-Macaulay, $m$ is equal to the length of $A_J/J^qA_J$. \cite[Proposition 7.1]{fulton}

We choose a regular sequence $r_1, \cdots, r_{\ell(w_0)-\ell(v)} \subset JA_J$ which generate $JA_J$ in $A_J$. (Note that $\Gamma'_F$, $W_{red}$ are of dimension $\ell(w_0), \ell(v)$, respectively.) We claim that $A_J/J^qA_J$ has the following composition series
$$0 \subset \langle r_1^{q-1}\rangle \subset \cdots \subset  \langle r_1\rangle \subset \langle r_1, r_2^{q-1}\rangle\subset \cdots \subset \langle r_1, r_2\rangle \subset \cdots \subset \langle r_1, r_2, \cdots, r_{\ell(w_0)-\ell(v)}\rangle \subset A_J/J^qA_J.$$
The inclusions are strict since $r_1, \cdots, r_{\ell(w_0)-\ell(v)} \subset JA_J$ is a regular sequence in any order. Also, successive quotients are isomorphic to $A_J/JA_J$, thus simple. Since the length of this chain is $q^{\ell(w_0)-\ell(v)}$, we are done.
\end{proof}

From the two propositions above we have
\begin{align*}
(id, F)^*[\cl{\O_w}] &= \sum_{\substack{u, v \in W \\ uwv^{-1}=w_0\\\ell(u)+\ell(v)=\ell(w_0)-\ell(w)}}(id, F)^*[\cl{C_{w_0u}} \times \cl{C_{w_0v}}]\\
&= \sum_{\substack{u, v \in W \\ uwv^{-1}=w_0\\\ell(u)+\ell(v)=\ell(w_0)-\ell(w)}}[\cl{C_{w_0u}}]\cdot[\cl{C_{w_0\act{F^{-1}}v}}]q^{\ell(v)}
\\&= \sum_{\substack{u, v \in W \\ uw\act{F}v^{-1}=w_0\\\ell(u)+\ell(v)=\ell(w_0)-\ell(w)}}[\cl{C_{w_0u}}]\cdot[\cl{C_{w_0v}}]q^{\ell(v)}.
\end{align*}
Thus we conclude the following.
\begin{thm}\label{dl:class} In $A_*(\B)$ we have 
$$[\cl{X(w)}] = \sum_{\substack{u, v \in W \\ uw\act{F}v^{-1}=w_0\\\ell(u)+\ell(v)=\ell(w_0)-\ell(w)}}[\cl{C_{w_0u}}]\cdot[\cl{C_{w_0v}}]q^{\ell(v)}.$$
\end{thm}
For $w=id$, $[X(id)] =  \sum_{v\in W}[\cl{C_{\act{F}v}}]\cdot[\cl{C_{w_0v}}]q^{\ell(v)}= \sum_{v \in W^F}q^{\ell(v)} = |G^F|/|B^F|$ which is  a special case of Proposition \ref{dl:comp}.

Now we have ingredients to strengthen Proposition \ref{dl:basis} as promised.

\begin{thm}\label{dl:realbasis} Let $M = |G^F|_{p'}$ be the largest factor of $|G^F|$ prime to $\ch \k=p$. Then the determinant (up to sign) of the transition matrix of bases from $\{[\cl{C_w}]\}_{w\in W}$ to $\{[\cl{X(w)}]\}_{w\in W}$ is invertible in $\Z[1/M]$. Therefore, $\{[\cl{X(w)}]\}_{w \in W}$ form a linear basis of $A_*(\B)_{\Z[1/M]}$.
\end{thm}
\begin{proof} Let $d$ be the determinant (up to sign) of the transition matrix of bases from $\{[\cl{C_w}]\}_{w \in W}$ to $\{[\cl{X(w)}]\}_{w \in W}$. As $\{[\cl{C_w}]\}_{w \in W}$ form a linear basis of $A_*(\B)$, the second statement follows from the first one. Since Proposition \ref{dl:basis} implies that $d$ is a unit in $\Z[1/N]$ where $N= q|G^F|$, it suffices to show that $p \nmid d$. To that end, we consider $A_*(\B)_{\F_p}$. By Theorem \ref{dl:class}, in this Chow group we have
\begin{align*}
[\cl{X(w)}] = \sum_{\substack{u, v \in W \\ uw\act{F}v^{-1}=w_0\\\ell(u)+\ell(v)=\ell(w_0)-\ell(w)}}[\cl{C_{w_0u}}]\cdot[\cl{C_{w_0v}}]q^{\ell(v)}= [\cl{C_{w^{-1}}}]\cdot[\cl{C_{w_0}}] = [\cl{C_{w^{-1}}}].
\end{align*}
Therefore $\{[\cl{X(w)}]\}_{w \in W}$ is the same basis of $A_*(\B)_{\F_p}$ as $\{[\cl{C_w}]\}_{w \in W}$ up to permutation. But it means that $p\nmid d$, hence the result.
\end{proof}

\section{An analogue of Deligne-Lusztig varieties corresponding to a regular semisimple element}
We obtain an analogue of a Deligne-Lusztig variety by replacing the geometric Frobenius morphism by another endomorphism on $G$. In this section we mainly focus on the variety using the conjugation by a regular semisimple element. Suppose $\s \in G$ is a fixed regular semisimple element and consider $\Y_{w,\s}$ following Definition \ref{def:analogue}. We will see that in many cases $\Y_{w,\s}$ can be considered as a limit of $X(w)$ as $q$ approaches 1.

In this section the characteristic of $\k$ is arbitrary. Recall that by \cite[Lemma 1.1]{lu:reflection} and its following argument, $\Y_{w,\s}$ is smooth of pure dimension $\ell(w)$ similar to a Deligne-Lusztig variety $X(w)$.

\subsection{The number of components of $\Y_{w,\s}$} We consider the number of (irreducible) components of $\Y_{w,\s}$. First we need a lemma.

\begin{lemma} $\Y_{w,\s}$ admits a stratification $\Y_{w,\s} = \sqcup_{w' \in W} \Y_{w,\s}^{w'}$ where
$$\Y_{w,\s}^{w'} = \{ B' \in \B \mid B \sim_{w'} B'\}.$$
Also, $\Y_{w,\s}^{w'}$ is isomorphic to $U^- \cap \act{w'^{-1}}U \cap BwB$.
\end{lemma}
\begin{proof} It follows from the proof of \cite[Proposition 1.2]{lu:reflection}.
\end{proof}

Since $\Y_{w,\s}$ is of pure dimension $\ell(w)$, the number of (irreducible) components of $\Y_{w,\s}$ is the same as the sum of the number of $\ell(w)$-dimensional irreducible components of $\Y_{w,\s}^{w'}$ over $w' \in W$. By the lemma above, it is the same as that of $U^- \cap \act{w'^{-1}}U \cap BwB$ over $w' \in W$.

To obtain this number, we first assume that $\ch \k = p \neq 0$. (Later we will see that $\ch \k =0$ case follows directly from it.) We may also assume thus $G$ is split over $\F_q$ where $q$ is some power of $p$ and $B$ is rational. Thus it is possible to count the number of $F$-fixed points of $U^- \cap \act{w'^{-1}}U \cap BwB$. Define
$$f_{w, w'}(q) \colonequals |(U^- \cap \act{w'^{-1}}U \cap BwB)^F|$$
and $f_{w}(q) \colonequals \sum_{w' \in W} f_{w, w'}(q)$. We will see in a moment that there is a polynomial $\f_w(x)$ of degree $\ell(w)$ independent of $q$ such that $f_w(q)=\f_w(q)$. Then the sum of $\ell(w)$-dimensional irreducible components is the same as the leading coefficient of $\f_w(x)$.

\begin{defn} \label{def:hecke}The Iwahori-Hecke algebra, denoted $\H_x(W)$, is the $\C[x, x^{-1}]$-algebra with a linear basis $\{T_w\}_{w \in W}$ with the following relations.
\begin{itemize}
\item For $w, w' \in W$ with $\ell(w) + \ell(w') = \ell(w w')$, $T_w T_{w'} = T_{ww''}$.
\item For $s \in S$, $(T_s -x)(T_s +1) =0$.
\end{itemize}
For $q \in \C^*$, define $\H_q(W)$ to be the $\C$-algebra with the same basis and relations with $x$  replaced by $q$.
\end{defn}

\begin{lemma} $f_{w, w'}(q)$ is the coefficient of $T_{w'^{-1}}$ in the expression of $T_w T_{w'^{-1}}$ with respect to the basis $\{T_v\}_{v \in W}$ in $\H_q(W)$.
\end{lemma}
\begin{proof} \cite[Theorem 2.6(b)]{kawanaka}.
\end{proof}
Note that in $\H_x(W)$ this gives a polynomial  of $x$, and $f_{w, w'}(q)$ is obtained by evaluating this polynomial at $q$.

As we assume $G$ is split, the corresponding geometric Frobenius $F$ acts trivially on $W$. We have a stratification of a Deligne-Lusztig variety $X(w) = \sqcup_{w' \in W} X^{w'}(w)$ where
$$X^{w'}(w) \colonequals \{ B' \in X(w) \mid B \sim_{w'} B'\}.$$
Note that $U^F \subset B^F$ acts on the left of each $X^{w'}(w)$.

\begin{lemma} $|(U^F \backslash X^{w'}(w))^F|$ is the coefficient of $T_{w'}$ in the expression of $T_w T_{w'}$ with respect to the basis $\{T_v\}_{v \in W}$ in $\H_q(W)$.
\end{lemma}
\begin{proof} \cite[Proposition 8.2]{dm:end}.
\end{proof}

As we sum up over $w' \in W$ we get the following.
\begin{prop} $f_w(q) = |(U^F\backslash X(w))^F|$.
\end{prop}
The following theorem is a natural consequence.
\begin{thm}\label{ss:comp} The number of components of $\Y_{w,\s}$ is the same as $|U^F \backslash G^F/P_I^F| = |W/W_I|$ where $I = \supp(w)$. In particular, $\Y_{w,\s}$ is irreducible if and only if $X(w)$ is irreducible.
\end{thm}
\begin{proof} We see that the number of components for $U^F \backslash X(w)$ and $\Y_{w,\s}$ coincides. Now the result follows from Lemma \ref{dl:irred} and \ref{dl:ind}.
\end{proof}
Note that $|G^F/P_I^F| = \sum_{wW_I \in W/W_I} q^{\ell(w)}$ where each $w$ is the unique right $I$-reduced element on each coset. Thus $|U^F \backslash G^F/P_I^F|$ is the same as the expression above where $q$ is replaced by 1. In other words, the number of components of $\Y_{w,\s}$ is the same as that of $X(w)$ for "$q=1$".

\begin{rmk} If $\ch \k =0$, then we may reduce $U^- \cap \act{w'^{-1}}U \cap BwB$ to $\cl{\F_p}$ for some $p$ with the number of its irreducible components unchanged. Then the result directly follows from above.
\end{rmk}

\subsection{Linear dependence of $[\cl{\Y_{w,\s}}]$ in the Chow group of a flag variety} Unlike $\{[\cl{X(w)}]\}$, unfortunately, in general $\{[\cl{\Y_{w,\s}}]\}$ does not form a basis in the Chow group of a flag variety. We give some examples when $[\cl{\Y_{w,g}}]$ coincides for different $w \in W$ and $g\in G$ not necessarily regular semisimple.

\begin{lemma}\label{inverse} $\Y_{w,g}=\Y_{w^{-1},g^{-1}}$. In particular, $[\cl{\Y_{w,g}}]=[\cl{\Y_{w^{-1},g^{-1}}}]$ in $A_*(\B)$.
\end{lemma}
\begin{proof}Indeed, we have
\begin{align*}
\Y_{w,g} &= \{ B' \in \B \mid B' \sim_w \act{g}B'\} = \{ B' \in \B \mid B' \sim_{w^{-1}} \act{g^{-1}}B'\} = \Y_{w^{-1},g^{-1}}.
\end{align*}
\end{proof}
For $\s$ regular semisimple, $[\cl{\Y_{w,\s}}]$ does not depend on $\s$ by the following lemma.
\begin{lemma}\label{ss:indep} $[\cl{\Y_{w,\s}}]$ is independent of the choice of a regular semisimple element $s \in G$.
\end{lemma}
\begin{proof} By the proof of \cite[Proposition 1.2]{lu:reflection}, an aforementioned variety $\Y_{w,\s}^{w'}$ defines a trivial fiber bundle on the set of regular semisimple elements in $T$, denoted $T^{rs}$. Therefore, any irreducible component of $\Y_{w,\s}$ is also generically trivial on $T^{rs}$. That is, for each $\s \in T^{rs}$ we can continuously choose each irreducible component of $\Y_{w,\s}$ on $T^{rs}$, and there exists a dense open subset $U_{\s}$ of the component such that $U_{\s}$ gives a trivial fiber bundle on $T^{rs}$. Since $T^{rs}$ is rational, thus rationally connected, we see that $[\cl{U_{\s}}]$ is independent of $\s\in T^{rs}$. As $\cl{\Y_{w,\s}}$ is the union of such $\cl{U_{\s}}$ for any $ \s\in T^{rs}$, the class $[\cl{\Y_{w,\s}}]$ is also independent of $\s \in T^{rs}$. Since in general $\act{h}\Y_{w, g} = \Y_{w, \act{h}g}$ for any $h \in G$ and $T^{rs}$ meets every conjugacy class of regular semisimple elements, we get the result.
\end{proof}
Thus we have
\begin{prop} \label{ss:winv}$[\cl{\Y_{w,\s}}] = [\cl{\Y_{w^{-1},\s}}]$.
\end{prop}
As a result, unless every element in $W$ is an involution, $\{[\cl{\Y_{w,\s}}]\}_{w \in W}$ are not linearly independent. There is another, more nontrivial criterion which identifies some classes of such.

\begin{prop} \label{ss:ww'} Let $w, w' \in W$ such that $\supp(w) \cap \supp(w') = \emptyset$. Then $[\cl{\Y_{ww',\s}}] = [\cl{\Y_{w'w,\s}}]$ for any regular semisimple $\s\in G$.
\end{prop} 
\begin{proof} Without loss of generality we assume $ \s\in T$ and consider a rational map $f:\P^1 \dashrightarrow T$ such that $f(1) = \s$ and $f(0) = id$. Denote by $U \subset \P^1$ the maximal open subset of $\P^1$ where $f$ is well-defined. Define
$$\bY_{w, w'} = \{ (B', B'', t) \in \B \times \B \times U \mid B' \sim_{w} B'', B'' \sim_{w'} \act{f(t)}B'\}.$$
Note that each fiber is isomorphic to $\Y_{ww', f(t)}$. We define $\bY_{w', w}$ similarly. 

The automorphism $\varphi: \B\times \B\times U \rightarrow \B \times \B\times U : (B', B'', t) \mapsto (B'', \act{f(t)}B', t)$ sends $\bY_{w, w'}$ isomorphically to $\bY_{w', w}$ with inverse $\varphi^{-1}: (B', B'', t) \mapsto (\act{f(t)^{-1}}B'', B', t)$, thus induces an isomorphism $\varphi : \cl{\bY_{w, w'}} \mapsto \cl{\bY_{w', w}}$ with inverse $\varphi^{-1}$, where the closure is taken in $\B\times \B \times U$. (Here we use $\supp(w) \cap \supp(w') = \emptyset,$ thus $\ell(w)+\ell(w') = \ell(ww') = \ell(w'w).$) 

Note that 
$$\cl{\bY_{w, w'}} \subset Z_{w, w'} \colonequals \{ (B', B'', t) \in \B \times \B \times U \mid \exists u \leq w, v \leq w' \text{ s.t. } B' \sim_{u} B'', B'' \sim_{v} \act{f(t)}B'\}$$
(similarly for $\cl{\bY_{w', w}}$) since the latter set is clearly closed and contains $\bY_{w, w'}$. We claim that $\varphi$ is the identity on $Z_{w, w'}\cap \B\times \B \times \{0\}$ (set-theoretically.) Indeed, 
\begin{align*}
Z_{w, w'}\cap \B\times \B \times \{0\} &= \{ (B', B'', 0) \in \B \times \B \times \{0\} \mid \exists u \leq w, v \leq w' \text{ s.t. } B' \sim_{u} B'', B'' \sim_{v} B'\}
\\&=\{(B', B', 0) \in \B \times \B \times \{0\}\}
\end{align*}
since $\supp(w) \cap \supp(w') = \emptyset$. Now it is clear that $\varphi$ acts as the identity on this set. Thus $\varphi$ also acts as the identity on $\cl{\bY_{w, w'}} \cap \B\times \B \times \{0\}$. In other words, set-theoretically $\cl{\bY_{w, w'}} \cap \B\times \B \times \{0\} = \cl{\bY_{w', w}} \cap \B\times \B \times \{0\}$. 

Since the isomorphism $\varphi : \cl{\bY_{w, w'}} \rightarrow \cl{\bY_{w', w}}$ preserves $t$, we see that the scheme-theoretic intersection of either $\cl{\bY_{w, w'}}$ or $\cl{\bY_{w', w}}$ with $\B\times \B \times \{t\}$ for any $t$ are also isomorphic. For $t=0$, it follows that every irreducible component of the intersection of either $\cl{\bY_{w, w'}}$ or $\cl{\bY_{w', w}}$ with $\B\times \B \times \{0\}$ has the same multiplicity on its scheme-theoretic intersection. In other words, the scheme-theoretic fiber of either $\cl{\bY_{w, w'}}$ or $\cl{\bY_{w', w}}$ at $t=0$ has the same class in $A_*(\B\times \B)$. We call it $[Z] \in A_*(\B\times \B)$.

Let $V \subset U$ be the set of $t \in U$ such that $f(t)$ is regular semisimple. We may assume that the set-theoretic fiber of $\cl{\bY_{w, w'}}$ or $\cl{\bY_{w', w}}$ at $1\in V$ is equivalent to its scheme-theoretic fiber, i.e. the intersections of $\cl{\bY_{w, w'}}$ or $\cl{\bY_{w', w}}$ and $\B\times \B\times\{1\}$ are transversal, since it is true for generic $t\in V$. (Indeed, it is easy to prove that it holds for any $t \in V$.) On the other hand, by the proof of \ref{ss:indep}, $\cl{\bY_{w, w'}}$ on $V$ is the closure of a trivial bundle. Thus the fiber of $\cl{\bY_{w, w'}}$ at $t \in V$ is the same as the closure of the fiber of $\bY_{w, w'}$ at $t$, which is isomorphic to
$$F_{w, w',f(t)} \colonequals \{ (B', B'') \in \B \times \B \mid \exists u \leq w, v\leq w' \text{ s.t. } B' \sim_u B'', B'' \sim_v \act{f(t)}B'\}.$$
(Similarly we define $F_{w', w, f(t)}$.)
Therefore, for $t=0,1$ we have
$$[F_{w, w', \s}] =[Z]= [F_{w', w, \s}]$$
on $A_*(\B\times \B).$

To conclude, let $\pi_1 : \B \times \B \rightarrow \B$ be the projection on the first factor. Clearly $(\pi_1)_*([F_{w, w', \s}]) = [\cl{\Y_{ww', \s}}]$ since $\pi_1$ restricts to a birational finite morphism from $F_{w, w', \s}$ to $\cl{\Y_{ww', \s}}$. Similarly we have $(\pi_1)_*([F_{w', w, \s}]) = [\cl{\Y_{w'w, \s}}]$. As $[F_{w, w', \s}] = [F_{w', w, \s}]$ we get the result.
\end{proof}

\begin{rmk} In general, for $w, w' \in W$ even if $\ell(w) + \ell(w') = \ell(ww') = \ell(w'w)$, $ [\cl{\Y_{ww', \s}}]$ might not equal $ [\cl{\Y_{w'w, \s}}]$. That is, condition for disjoint supports is essential for the proof of this proposition.
\end{rmk}

\subsection{$[\cl{\Y_{w,\s}}]$ in terms of Schubert classes} In a similar way to a Deligne-Lusztig variety, we calculate $[\cl{\Y_{w,\s}}]$ in terms of Schubert classes. To that end, we let $\Gamma_\s$ be the graph of $\ad(\s)$ in $\B \times \B$ where $\s\in G$ is again regular semisimple. We start with the following lemma.
\begin{lemma} The intersection of $\Gamma_\s$ with $\cl{\O_w}$ is generically transversal. More precisely, the intersection of $\Gamma_\s$ with $\O_w$ is (nonempty and) transversal.
\end{lemma}

\begin{proof} \cite[Lemma 1.1]{lu:reflection}
\end{proof}

Therefore similarly to Deligne-Lusztig varieties, we have $[\cl{\Y_{w,\s}}] = (id, \ad(\s))^*([\cl{\O_w}])$. Since $\ad(\s)$ is an automorphism on $\B$ which acts trivially on $A_*(\B)$, it is equivalent to $(id, id)^*([\cl{\O_w}])$. By Proposition \ref{orbitdecomp}, it is equal to
$$(id, id)^*\left(\sum_{\substack{u, v \in W \\ uwv^{-1}=w_0\\\ell(u)+\ell(v)=\ell(w_0)-\ell(w)}}[\cl{C_{w_0u}} \times \cl{C_{w_0v}}]\right) = \sum_{\substack{u, v \in W \\ uwv^{-1}=w_0\\\ell(u)+\ell(v)=\ell(w_0)-\ell(w)}}[\cl{C_{w_0u}}] \cdot [\cl{C_{w_0v}}].$$
Thus we obtain the following.
\begin{thm} For $\s\in G$ regular semisimple and for $w\in W$, we have 
$$[\cl{\Y_{w,\s}}] =  \sum_{\substack{u, v \in W \\ uwv^{-1}=w_0\\\ell(u)+\ell(v)=\ell(w_0)-\ell(w)}}[\cl{C_{w_0u}}] \cdot [\cl{C_{w_0v}}].$$
\end{thm}
Observe that it is equal to $[\cl{X(w)}]$ if $F$ acts trivially on $W$ and $"q=1"$.

\subsection{Class of an irreducible component of $\cl{\Y_{w,\s}}$} $\cl{\Y_{w,\s}}$ is not in general irreducible by Theorem \ref{ss:comp}. Thus it is natural to ask whether the class of each irreducible component of $\cl{\Y_{w,\s}}$ is the same. We claim that this is indeed the case. Let $\s \in T$, $w\in W$ and $I = \supp(w)$. Choose $\{u_j\in W\}_{j\in J}$ such that $\{u_jW_I\}_{j\in J}$ becomes the complete collection of right $W_I$-cosets. For each $u_j$, we define
$$\rho_j: \Y_{L_I, w, \act{u_j^{-1}}\s} \rightarrow \Y_{w,\s}: B'_I \mapsto \act{u_j}(B'_I U_I),$$
where $\Y_{L_I, w, \act{u_j^{-1}}\s}$ is defined in the same way corresponding to the Levi subgroup $L_I$ of the parabolic subgroup $P_I$, with the maximal torus $T \subset L_I$ and the Borel subgroup $B_I = L_I\cap B$. (This defines a natural injection $W_I \subset W$ of Coxeter groups and $w$ is considered an element of $W_I$.) Also $U_I$ is the unipotent radical of $P_I$.
\begin{lemma} \label{ss:push}The map above is well-defined and injective.
\end{lemma}
\begin{proof} Note that for a Borel subgroup $B_I'$ of $L_I$, $\act{u_j}(B'_I U_I)$ gives a Borel subgroup of $G$. If $B_I' \sim_{w} \act{\act{u_j^{-1}}\s}B_I'$, then there exists $l \in L_I$ such that $B_I' = \act{l}B_I$ and $\act{\act{u_j^{-1}}\s}B_I' = \act{lw}B_I$. Since $L_I$ normalizes $U_I$, $\act{u_j}(B'_I U_I)= \act{u_j}(\act{l}B_I U_I) = \act{u_j l}(B_IU_I) = \act{u_j l}B.$ Also,
$$\act{\s}(\act{u_j}(B'_I U_I)) = \act{u_j\act{u_j^{-1}}\s}(B'_I U_I) = \act{u_j}(\act{\act{u_j^{-1}}\s}B'_IU_I) = \act{u_j}(\act{lw}B_I U_I) = \act{u_j lw}(B_IU_I) = \act{u_j lw}B$$
since $l$ and $w$ normalizes $U_I$. Therefore the map $\rho_j$ is well-defined. Injectivity of $\rho_j$ follows from that $\act{u_j^{-1}}\rho_j(B_I')\cap L_I = B_I'.$
\end{proof}
Consider the coproduct of such morphisms
$$\rho\colonequals\sqcup_{j \in J} \rho_j : \bigsqcup_{j \in J}  \Y_{L_I, w, \act{u_j^{-1}}\s} \rightarrow \Y_{w,\s}.$$
\begin{lemma} \label{domin}The morphism above is injective and dominant.
\end{lemma}
\begin{proof} In order to prove $\rho$ is injective, note that for a Borel subgroup $B'_I$ of $L_I$, $\act{u_j}(B'_I U_I)$ is a Borel subgroup of $\act{u_j}P_I$. Thus for different $j$, the image of $\rho_j$ is disjoint to one another. Combined with the lemma above it follows that $\rho$ is injective. That $\rho$ is dominant follows form the fact that both $\bigsqcup_{j \in J}  \Y_{L_I, w, \act{u_j^{-1}}\s}$ and $\Y_{w,\s}$ are of pure dimension $\ell(w)$ with the same number of irreducible components.
\end{proof}

Let $\mathbb{Y}_w\rightarrow T^{rs}$ be the closure of the whole space of the fiber bundle $\Y_{w,\s}$ over $T^{rs}$, the set of regular semisimple elements in $T$. From the proof of Lemma \ref{ss:indep}, it is a closure of a trivial bundle on $T^{rs}$, thus $\mathbb{Y}_w$ is equal to the whole space of the fiber bundle $\cl{\Y_{w,\s}}$ and we may write
$$\mathbb{Y}_w = \bigcup_{j \in J} \mathbb{Y}_{w, j}$$
such that the fiber of each $\mathbb{Y}_{w, j}$ at $\s\in T^{rs}$ is each irreducible component of $\cl{\Y_{w,\s}}$.

Also note that the image of $\rho$ consists of Borel subgroups which is also a Borel subgroup of $\act{u_j}P_I$ for some $j \in J$. Since this is a closed condition, by Lemma \ref{domin} any $B' \in \cl{\Y_{w,\s}}$ also satisfies the same property. However, it is impossible for a Borel subgroup to be contained in $\act{u_j}P_I$ for different $j$'s. Thus we have a bijection
$$\{\text{irreducible components of }\cl{\Y_{w,\s}}\} \leftrightarrow \{\act{u_j}P_I\}_{j\in J}.$$
Since this is true for all $ \s\in T^{rs}$, we may reorder $\{\mathbb{Y}_{w,j}\}_{j\in J}$ such that every $(\s, B' \in \cl{\Y_{w,\s}}) \in \mathbb{Y}_{w,j}$ satisfies $B' \subset \act{u_j}P_I$.

For any $w' \in W$ we have a morphism
$$\ad(w') : \Y_{w,\s} \xrightarrow{\simeq} \Y_{w, \act{w'}\s},$$
which clearly induces an automorphism on $\mathbb{Y}_w$. Therefore there is a natural $W$-action on $\mathbb{Y}_w$ which permutes $\{\mathbb{Y}_{w,j}\}_{j\in J}$. We claim that $W$ permutes them transitively. Indeed, for $B' \subset {\act{u_j}}P_I$, we have $\act{w'}B' \subset \act{w'u_j}P_I$ and $W$ acts on $\{\act{u_j}P_I\}_{j\in J}$ transitively. Since for each $j$ a fiber of $\mathbb{Y}_{w,j}$ at any $\s\in T^{rs}$ gives the same class in $A_*(\B)$ (proof of Lemma \ref{ss:indep}) and the adjoint action of $G$ on $\B$ is trivial on $A_*(\B)$, we have the following theorem.
\begin{thm} \label{ss:same}The class of every irreducible component of $\cl{\Y_{w,\s}}$ is the same in $A_*(\B)$.
\end{thm}
\begin{rmk} For a Deligne-Lusztig variety, this is clear since $G^F$ acts transitively on the irreducible components of $\cl{X(w)}$. However, we do not have a direct action of $W$ on $\Y_{w,\s}$ per se, which makes the proof slightly complicated.
\end{rmk}

\section{An analogue of Deligne-Lusztig varieties corresponding to a regular unipotent element}
It is also reasonable to consider $\Y_{w,\u}$ for $\u\in G$ regular unipotent. In this section $\ch \k$ is good, i.e. it does not divide the coefficients of the highest root of the root system of $G$. Unlike $X(w)$ or $\Y_{w,\s}$, we will see that $\Y_{w,\u}$ is in general not smooth.

\subsection{Some pathologies of $\Y_{w,\u}$} On \cite[p. 550]{kawanaka}, the author asked the following question:
\begin{quote}
- Is $\Y_{w,\u}$ isomorphic to an $\ell(w)$-dimensional affine space?
\end{quote}
This is based on the following fact that if $\Y_{w,\u}$ is defined over $\F_q$, the number of $F$-fixed points is $q^{\ell(w)}$. Obviously we have
\begin{lemma} \label{unip:dim} $\dim \Y_{w,\u} = \ell(w).$
\end{lemma}
\begin{proof} If $\ch \k \neq 0$, then by \cite[(7.1)]{kawanaka} we have $|(\Y_{w,\u})^F| = q^{\ell(w)}$ where $F$ is the geometric Frobenius corresponding to $\F_q$, thus the result is obvious. Otherwise, since all regular unipotent elements are conjugate to one another $\Y_{w,\u}$ is isomorphic for any $\u \in G$ regular unipotent. We choose $\u \in G$ such that reduction of $\Y_{w,\u}$ to $\cl{\F_p}$ for some $p$ is well-defined and has the same dimension as the original one. Thus it is also clear.
\end{proof}

However, we give an example that the answer of the question above is in general negative.
\begin{example} Consider $G=GL_3$ over $\k=\C$ (thus $W=S_3$) and let $B$ be the set of upper triangular invertible matrices and $T$ be the set of diagonal invertible matrices. Also we let
$$\u=\left (
  \begin{array} {ccc} 1 & 1 & 0 \\ 0 & 1 & 1 \\ 0 & 0 & 1 \\
  \end{array} \right).$$
Let $\B$ be the set of complete flags in $\C^3$. For any $\cF \in \B$, we choose $\vec{v}_1=(v_{11}, v_{12}, v_{13})$ and $\vec{v}_2=(v_{21}, v_{22}, v_{23})$ such that $\cF = [0 \subset \langle \vec{v}_1 \rangle \subset \langle \vec{v}_1, \vec{v}_2 \rangle \subset \C^3]$. We have a Plucker embedding $\B \hookrightarrow \P^2 \times \P^2$ which sends $\cF$ to $([v_{11}, v_{12}, v_{13}], [v_{11}v_{22}-v_{21}v_{12}, v_{11}v_{23}-v_{21}v_{13}, v_{12}v_{23}-v_{22}v_{13}])$. Note that this is independent of the choice of $\vec{v}_1, \vec{v}_2$.

Then for $w_0=s_1s_2s_1 \in W$, we may identify
$$\Y_{w_0, \u}=\{([a, b, c], [d, e, f]) \in \P^2 \times \P^2 \mid af-be+cd=0, bf-ce\neq 0\text{ or }cf\}.$$
Define
$$Z = \{([b, c], [e, f]) \in \P^1 \times \P^1 \mid bf-ce\neq 0 \text{ or } cf\}.$$
Then there is a morphism $\pi : \Y_{w_0, \u} \rightarrow Z : ([a, b, c], [d, e, f]) \mapsto ([b, c], [e, f]).$ Also this is an $\A^1$-torsor with the action $\A^1$ on $\Y_{w_0, \u}$ as follows.
$$t \in \A^1_\k, \quad t\cdot([a, b, c], [d, e, f]) =([a+ct, b, c], [d-ft, e, f])$$
Thus in particular, $\Pic(\Y_{w_0, \u}) = \Pic(Z)$. Now let $C, D \in \Pic(\P^1 \times \P^1)$ be the generators of $\Pic(\P^1 \times \P^1)$ which correspond to $[\P^1 \times \{*\}]$, $[\{*\} \times \P^1]$, respectively. Then the complement of $Z$ in $\P^1 \times \P^1$ is the union of two hyperplanes of the same class $C+D \in \Pic(\P^1 \times \P^1)$. Thus $\Pic(Z) \simeq \Z$. But it implies $\Y_{w_0, \u}$ cannot be an affine space.
\end{example}

However, there are some special cases where it is true. 
\begin{prop}\label{unip:homog} Suppose $w \in W$ is elliptic (i.e. has no eigenvalue 1 on the reflection representation of $W$)  and has a minimal length among its conjugates. Then $\Y_{w,\u}$ is isomorphic to $\A^{\ell(w)}$.
\end{prop}
\begin{proof} By \cite[0.3(a)]{lu:homogeneity}, the centralizer $Z_G(\u)$ of $\u\in G$ acts transitively on $\Y_{w,\u}$. Since the centre of $G$ acts trivially on $\B$, we see that the centralizer $Z_U(\u)$ in $U$, isomorphic (as a variety) to an affine space, acts transitively on $\Y_{w,\u}$. (Recall that we assume $\ch \k$ is good, thus $Z_U(\u)$ is connected.) Since $Z_U(\u)$ is abelian, $\Y_{w,\u}$ is isomorphic to a quotient group of $Z_U(\u)$ as a variety, which is also an affine space. Now the result follows from Lemma \ref{unip:dim}.
\end{proof}

In general, $\Y_{w,\u}$ is not quasi-affine. Furthermore, it is neither smooth, normal, nor even rationally smooth. We give an example for this pathology.
\begin{example} Let $G=GL_4$ over $\k=\C$, thus $W=S_4$. As before let $B$ be the set of upper triangular invertible matrices and $T$ be the set of diagonal matrices. Also we let
$$\u=\left(
\begin{array}{cccc}
 1 & 1 & 0 & 0 \\
 0 & 1 & 1 & 0 \\
 0 & 0 & 1 & 1 \\
 0 & 0 & 0 & 1 \\
\end{array}
\right).$$
We consider $\Y_{w,\u}$ where $w=s_2s_1s_3s_2$. For $\cF \in \B$, choose $\vec{v}_i = (v_{i1}, v_{i2}, v_{i3}, v_{i4})  \in \C^4$ for $i=1,2,3$ such that $\cF = [0 \subset \langle \vec{v}_1 \rangle \subset \langle \vec{v}_1, \vec{v}_2 \rangle \subset \langle \vec{v}_1, \vec{v}_2, \vec{v}_3 \rangle \subset \C^4]$. We have a Plucker embedding $\B \hookrightarrow \P^3 \times \P^5 \times \P^3$ which sends $\cF$ to
$$\left(
\begin{aligned}
&\left[(v_{11}, v_{12}, v_{13}, v_{14})\right], 
\\&\left[\det\begin{psmallmatrix}v_{11}&v_{12}\\v_{21}&v_{22}\end{psmallmatrix},
\det\begin{psmallmatrix}v_{11}&v_{13}\\v_{21}&v_{23}\end{psmallmatrix},
\det\begin{psmallmatrix}v_{11}&v_{14}\\v_{21}&v_{24}\end{psmallmatrix},
\det\begin{psmallmatrix}v_{12}&v_{13}\\v_{22}&v_{23}\end{psmallmatrix},
\det\begin{psmallmatrix}v_{12}&v_{14}\\v_{22}&v_{24}\end{psmallmatrix},
\det\begin{psmallmatrix}v_{13}&v_{14}\\v_{23}&v_{24}\end{psmallmatrix}\right],
\\&\left[\det\begin{psmallmatrix}v_{11}&v_{12}& v_{13}\\v_{21}&v_{22}&v_{23}\\v_{31}&v_{32}&v_{33}\end{psmallmatrix},
\det\begin{psmallmatrix}v_{11}&v_{12}& v_{14}\\v_{21}&v_{22}&v_{24}\\v_{31}&v_{32}&v_{34}\end{psmallmatrix},
\det\begin{psmallmatrix}v_{11}&v_{13}& v_{14}\\v_{21}&v_{23}&v_{24}\\v_{31}&v_{33}&v_{34}\end{psmallmatrix},
\det\begin{psmallmatrix}v_{12}&v_{13}& v_{14}\\v_{22}&v_{23}&v_{24}\\v_{32}&v_{33}&v_{34}\end{psmallmatrix}\right]
\end{aligned}
\right)
$$
which is independent of the choice of $\vec{v}_i$. By this embedding we identify
\begin{align*}
\B= \{&([a_1, a_2, a_3, a_4],[b_1, b_2, b_3, b_4, b_5, b_6],[c_1, c_2, c_3, c_4]) \in \P^3 \times \P^5 \times \P^3 \mid 
\\& b_3b_4 - b_2b_5 + b_1b_6=0,  a_3b_1 - a_2b_2 + a_1b_4=0,
 a_4b_1 - a_2b_3 + a_1b_5=0,
 \\& a_4b_2 - a_3b_3 + a_1b_6=0,
 a_4b_4 - a_3b_5 + a_2b_6=0,
 b_3c_1 - b_2c_2 + b_1c_3=0,
 \\& b_5c_1 - b_4c_2 + b_1c_4=0,
 b_6c_1 - b_4c_3 + b_2c_4=0,
 b_6c_2 - b_5c_3 + b_3c_4=0,
 \\& c_1a_4-c_2a_3+c_3a_2-c_4a_1=0\}.
\end{align*}
Then we have
\begin{align*}
\Y_{w,\u} = \{&([a_1, a_2, a_3, a_4],[b_1, b_2, b_3, b_4, b_5, b_6],[c_1, c_2, c_3, c_4]) \in \B \mid
\\& a_4 c_3 - a_3 c_4 + a_4 c_4=0, a_4 c_2 - a_3 c_3 + a_2 c_4=0,  b_5^2\neq b_3 b_6 + b_4 b_6\}.
\end{align*}
Now consider $f:\P^1 \rightarrow \Y_{w,\u}$ given by
$$f([s,t]) = ([0, t, s, s], [0,0,0,1,1,0],[-t,-t,0, s]).$$
Direct calculation shows that this is well-defined. Since $f$ is an embedding, $\P^1 \subset \Y_{w,\u}$. But it means that $\Y_{w,\u}$ is not quasi-affine.

On the other hand, we consider the normalization of $\Y_{w,\u}$, say $\pi: \hat{\Y}_{w,\u} \rightarrow \Y_{w,\u}$. Direct calculation shows that $\pi$ is not an isomorphism, thus $\Y_{w, \u}$ is not normal. Furthermore, $\hat{\Y}_{w,\u}$ is indeed smooth, thus $\pi$ is a resolution of singularity. (This is true when $\text{char } \k \neq 2$.) Also on $Z \subset \Y_{w,\u}$ defined by
$$Z = \{([a_1, a_2, a_3, a_4],[b_1, b_2, b_3, b_4, b_5, b_6],[c_1, c_2, c_3, c_4]) \in \Y_{w,\u} \mid a_3=a_4 = b_6=c_3=c_4=0\},$$
there exists a two-dimensional subvariety $Z' \subset Z$ such that $\pi^{-1}(Z') \rightarrow Z'$ is a 2-1 map. (Indeed $Z$ is the singular locus of $\Y_{w,\u}$.) Thus we see that the constant sheaf $\C_{\Y_{w,\u}}$ is not equal to $IC_{\Y_{w,\u}}$, which means that $\Y_{w,\u}$ is not rationally smooth, thus not smooth as well.
\end{example}

\subsection{Irreducibility of $\Y_{w,\u}$} Unlike $X(w)$ or $\Y_{w,\s}$ for $\s\in G$ regular semisimple, $\Y_{w,\u}$ is always irreducible, which we prove in this section. Let $\U$ be the set of unipotent elements in $G$, which is a closed subvariety of $G$. We have the following lemma.
\begin{lemma} For any $w \in W$, $\U \cap \cl{BwB}$ is irreducible of dimension $\ell(w) +\dim U$.
\end{lemma}
\begin{proof} As before, it suffices to show for $G$ defined over $\F_q$ for $q$ some power of prime $p\neq 0$. Also we may assume $F$ acts trivially on $W$ and $B$ is rational. Then we have
\begin{align*}
|(\U \cap BwB)^F| &= |(U^- \cap BwB)^F||U^F| &\text{ \cite[Corollary 4.2.]{kawanaka}}
\\&=(T_{w}T_{w_0}:T_{w_0})|U^F| &\text{ \cite[Theorem 2.6.(b)]{kawanaka}}
\\&=(-q)^{\ell(w)}(T_{w^{-1}}^{-1}:1)|U^F| &\text{ \cite[Theorem 2.14.(c)]{kawanaka}}
\end{align*}
where $T_w \in \H_q(W)$ as in Definition \ref{def:hecke} and $(A:B)$ denotes the coefficient of $B$ in the expression of $A$. Note that $(-q)^{\ell(w)}(T_{w^{-1}}^{-1}:1)$ is equal to $R_{1, w}$, the R-polynomial in \cite[Section 2]{kl:hecke}. It is known that $R_{1,w}$ is a monic polynomial of degree $\ell(w)$, thus $|(\U \cap BwB)^F|$ is a monic polynomial of degree $\ell(w) +\dim U$.

On the other hand, $\U$ and $\cl{BwB}$ are both irreducible and their codimension in $G$ is $r$ and $\ell(w_0) - \ell(w)$, respectively, where $r$ is the rank of $G$. Thus each component of $\U \cap \cl{BwB}$ has codimension $\leq r + \ell(w_0) - \ell(w)$, or dimension $\geq  \ell(w)+\dim U$ because $G$ is smooth. But we also have $|(\U \cap \cl{BwB})^F|= \sum_{v \leq w} |(\U \cap BvB)^F|$ is a monic polynomial of degree $\ell(w) +\dim U$. From this we conclude that $\U \cap \cl{BwB}$ is irreducible of dimension $\ell(w) +\dim U$ as desired.
\end{proof}

Let $\cC$ be the set of regular unipotent elements in $G$, which is a single conjugacy class. Then clearly $\cC \cap BwB \subset \U \cap \cl{BwB}$, thus $\cl{\cC \cap BwB} \subset \U \cap \cl{BwB}$. Now fix $\u \in \cC \cap B$ and consider the following diagram.
$$G/B \xleftarrow{\quad\pi\quad} G \xrightarrow{\rho: g \mapsto \act{g}\u}\cC.$$
Then $\pi$ and $\rho$ are fiber bundles and we have $\pi^{-1}(\Y_{w,\u}) = \rho^{-1}(\cC \cap BwB)$. Thus in particular $\cC \cap BwB$ is nonempty and of dimension $\ell(w) + \dim B - \dim r = \ell(w) + \dim U$. In other words, $\cC \cap BwB$ is dense in $\U \cap \cl{BwB}$, which we state as follows.
\begin{lemma}$\cl{\cC \cap BwB} = \U \cap \cl{BwB}$.
\end{lemma}
In particular, $\cC \cap BwB$ is irreducible, and its closure in $\cC$ is the same as $\cC \cap \cl{BwB} = \bigcup_{v \leq w} \cC \cap \cl{BvB}$. As $\pi$ and $\rho$ are fiber bundles, we have $\pi^{-1}(\cl{\Y_{w,\u}}) = \rho^{-1}(\cl{\cC \cap BwB})$ which means the following.
\begin{prop}\label{unip:irred} For $ w\in W$, $\Y_{w,\u}$ is irreducible. Its closure is $\cl{\Y_{w,\u}} = \bigcup_{v \leq w} \Y_{v, \u}$.
\end{prop}
\subsection{$[\cl{\Y_{w,\u}}]$ in terms of Schubert classes}
We prove that $[\cl{\Y_{w,\u}}]$ is indeed the same as the class of an irreducible component of $\cl{\Y_{w,\s}}$. First we prove this when $\supp(w) = S$, i.e. $\Y_{w,\u}$ and $\Y_{w,\s}$ are both irreducible.
\begin{prop}\label{unip:trans} Let $\u \in G$ be a regular unipotent element, and $\Gamma_\u$ be the graph of $\ad(\u) : \B \rightarrow \B$. Then for $w\in W$ satisfying $\supp(w) = S$, the intersection of $\Gamma_\u$ and $\cl{\O(w)}$ is generically transversal.
\end{prop}
As a corollary we have
\begin{cor}\label{unip:ssequiv} If $\supp(w) = S$, then 
$$[\cl{\Y_{w,\u}}] =[\cl{\Y_{w,\s}}]=\sum_{\substack{u, v \in W \\ uwv^{-1}=w_0\\\ell(u)+\ell(v)=\ell(w_0)-\ell(w)}}[\cl{C_{w_0u}}] \cdot [\cl{C_{w_0v}}].$$
\end{cor}
\begin{proof} By Proposition \ref{unip:trans}, $[\cl{Y_{w,n}}] = (id, ad(\u))^*[\cl{\O(w)}]$. Since $ad(\u) : \B \rightarrow \B$ is an automorphism which acts trivially on $A_*(\B)$, we get the result.
\end{proof}
\begin{proof}[Proof of \ref{unip:trans}] (Assuming non-degeneracy of the Killing form on $\g$.) Let $(B', \act{\u}B') \in \Gamma_\u\cap \O(w)$. Then the tangent space of $\B \times \B$ at $(B', \act{\u}B')$ is isomorphic to $\g/\b' \times \g/\act{\u}\b'$ where $\g,\b'$ are the Lie algebras of $G$, $B'$, respectively. Let $\pi_1: \g \rightarrow \g/\b'$ and $\pi_2: \g \rightarrow \g/\act{\u}\b'$. The tangent space of $\Gamma_\u$ at $(B', \act{\u}B')$ is $\{(\pi_1(\vec{v}), \pi_2(\act{\u}\vec{v})) \in \g/\b' \times \g/\act{\u}\b' \mid \vec{v} \in \g\}$. Likewise, the tangent space of $\O(w)$ is $\{(\pi_1(\vec{v}'), \pi_2(\vec{v}')) \in \g/\b' \times \g/\act{\u}\b' \mid \vec{v}' \in \g\}$. Thus if they meet at $(B', \act{\u}B')$ transversally, then for any $\vec{v}_1, \vec{v}_2 \in \g$ there exists $\vec{v}, \vec{v}' \in \g$ such that the following holds.
$$\vec{v} + \vec{v}' = \vec{v}_1 \mod \b', \qquad \vec{v} + \act{\u}\vec{v}' = \vec{v}_2 \mod \act{\u}\b'$$
Or equivalently we need to find $\vec{v} \in \g$ such that
$$\vec{v} - \act{\u^{-1}}\vec{v} = \vec{v_1}- \act{\u^{-1}}\vec{v_2} \mod \b'$$
Thus we see that the transversality condition is equivalent to
$$\textup{im}(\ad(\u^{-1}) - id) + \b' = \g.$$

Let $\vec{x} \in \g$ be orthogonal to $\textup{im}(\ad(\u^{-1}) - id) + \b'$ with respect to the Killing form on $\g$. Then for any $\vec{v} \in \g$ we have
$$(\act{\u^{-1}}\vec{v}-\vec{v} , \vec{x}) = 0 \Leftrightarrow (\vec{v}, \vec{x}) = (\vec{v}, \act{\u}\vec{x}) \Leftrightarrow \act{\u}\vec{x}=\vec{x}.$$
Here we use non-degeneracy of the Killing form. Also $(\b', x) =0$ is equivalent to $x \in \n'$, the nilpotent radical of $\b'$. Thus $\textup{im}(\ad(\u^{-1}) - id) + \b' = \g$ if and only if $\n' \cap Z_\g(\u) = 0$ where $Z_\g(\u)$ is the set of elements in $\g$ fixed by $\ad(\u)$.

Let $U'$ be the unipotent radical of $B'$ and $Z_G(\u)$ be the centralizer of $\u$ in $G$. Since the condition above is equivalent to that $U' \cap Z_G(\u)$ is finite, clearly true if $B' \cap Z_U(\u) = \{id\}$ where $Z_U(\u)$ is the centralizer of $\u$ in $U$. Now suppose $w \in W$ is a Coxeter element. Then it is clearly elliptic and $\ell(w)$ is minimal among its conjugates, which is the same as the dimension of $Z_U(\u)$. Then the stabilizer of $B'$ by $Z_U(\u)$, or $B' \cap Z_U(\u)$, is isomorphic to an affine space of dimension 0, thus trivial. Thus in this case $\Gamma_\u$ and $\O(w)$ intersect transversally, which means $\Gamma_\u$ and $\cl{\O(w)}$ intersect generically transversally.

For general $w\in W$ with $\supp(w)=S$, there exists $w' \leq w$ where $w'$ is a Coxeter element. Then $\Gamma_\u \cap \O(w') \subset \Gamma_\u \cap \cl{\O(w)}$ and $\Gamma_\u \cap \cl{\O(w)}$ is irreducible by Proposition \ref{unip:irred}. Since $B' \cap Z_U(\u)=\{id\}$ is an open condition satisfied on $\Gamma_\u \cap \O(w')$, this is generically true on $\Gamma_\u \cap \cl{\O(w)}$. But this implies the desired statement by argument above.
\end{proof}

\begin{rmk} The first proof of \cite[Corollary 5.6]{lu:weyltounip} for a Coxeter element $w\in W$ and any $g\in G$ does not require that the Killing form on $\g$ is nondegenerate. As we only used nondegenracy for the proof in the case of a Coxeter element, Proposition \ref{unip:trans} is still true without this assumption.
\end{rmk}

\begin{lemma} Suppose $I = \supp(w)$. Let $P_I$ be the parabolic subgroup corresponding to $I$ which contains $B$. Let $L_I$ be the Levi subgroup of $P_I$ which contains $T$ and $U_I$ be the unipotent radical of $P_I$. Let $\u_I \in L_I \cap U$ be a regular unipotent element in $L_I$, and $\u' \in U_I$ such that $\u = \u'\u_I \in U$ is regular unipotent in $G$. (This is always possible.) Let $\rho : L_I / B_I \hookrightarrow G/B$ be the natural closed embedding. Then $\rho(\cl{\Y_{L_I,w, \u_I}}) = \cl{\Y_{w,\u}}$ where $\Y_{L_I, w, \u_I}$ is defined using $L_I, w\in W_I, \u_I \in L_I$.
\end{lemma}
\begin{proof} Suppose $B'_I \in \Y_{L_I, w, \u_I}$, i.e. $B'_I \sim_w \act{\u_I}B'_I$. Then there exists $l \in L_I$ such that $B'_I = \act{l}B_I$ and $\act{\u_I}B'_I = \act{lw}B_I$. Now 
$$\rho(B'_I) = B'_IU_I = \act{l}B_IU_I = \act{l}(B_IU_I) = \act{l}B$$
and also
$$\act{\u}(\rho(B'_I)) = \act{\u}(B'_IU_I)=\act{\u'\u_I}(B'_IU_I) = \act{\u'}(\act{\u_I}B'_IU_I) = \act{\u'}(\act{lw}B_IU_I) =  \act{\u'}\act{lw}B.$$
Since $l \in L_I$ normalizes $U_I$, $l^{-1}\u'lw \in BwB$, which means $\rho(B'_I) \sim_{w}\act{\u}(\rho(B'_I))$. Therefore $\rho(\Y_{L_I, w, \u_I}) \subset \Y_{w,\u}$. Since $\rho$ is a closed embedding and $\Y_{L_I, w, \u_I}$ and $\Y_{w,\u}$ are both irreducible of the same dimension, we get the result.
\end{proof}

In particular, we have $\rho_*[\cl{ \Y_{L_I, w, \u_I}}] = [\cl{\Y_{w,\u}}]$. But Lemma \ref{ss:push} for $u_j = id$ shows that for regular semisimple $\s \in T$, $\rho(\cl{ \Y_{L_I, w, \s}})$ is the same as an irreducible component of $\cl{ \Y_{w,\s}}$. Thus $\rho_*[\cl{ \Y_{L_I, w, \s}}]$ is equal to the class of any irreducible component of $[\cl{ \Y_{w,\s}}]$ by Theorem \ref{ss:same}. Since $[\cl{ \Y_{L_I, w, \u_I}}] = [\cl{ \Y_{L_I, w, \s}}]$ by Corollary \ref{unip:ssequiv}, we have the following.
\begin{thm} \label{unip:class} For $I  =\supp(w)$ we have 
$$[\cl{\Y_{w,\u}}] = \frac{|W_I|}{|W|}[\cl{ \Y_{w,\s}}] = \frac{|W_I|}{|W|}\sum_{\substack{u, v \in W \\ uwv^{-1}=w_0\\\ell(u)+\ell(v)=\ell(w_0)-\ell(w)}}[\cl{C_{w_0u}}] \cdot [\cl{C_{w_0v}}].$$
\end{thm}

\begin{rmk} Here is another strategy to prove Theorem \ref{unip:class}. Indeed, since $G$ is rational and the set of regular semisimple elements are open dense, we choose a parametrization $f: \P^1 \dashrightarrow G$ such that $f(1)=\s$ regular semisimple and $f(0)=\u$ regular unipotent. Then it is easy to show that $[\cl{\Y_{w,\s}}]$ and $[\cl{\Y_{w,\u}}]$ are parallel in $A_*(\B)_{\Q}$. However, it is a little subtle to compute the multiplicity of the fiber at 0, which now we know from Theorem \ref{unip:class} is equal to $|W|/|W_I|$ where $I = \supp(w)$.
\end{rmk}

\section{Example: type A}
For type A, the structure of $A_*(\B)$ is well-known; if $G$ is of type A$_{n-1}$ we have (and fix throughout this section) an isomorphism of rings 
$$A_*(\B) \rightarrow \Z[x_1, \cdots, x_n]/J: [\cl{C_w}]\mapsto \fS_{w_0w}(x)$$
 where $J$ is an ideal generated by symmetric functions and $\fS_{w_0w}(x)$ is called a Schubert polynomial. For more information one may refer to \cite{manivel} or \cite{fulton:schubert}. Here we assume that readers are familiar with this theory.

We reformulate our results in terms of Schubert polynomials. For simplicity, assume that $\ch \k$ is good for $G$ throughout this section.

\begin{thm} Let $G$ be of type A${_{n-1}}$ and identify $W=S_n$. If $G$ is defined over $\F_q$ for $q$ power of $\ch \k$, let $F$ be the geometric Frobenius morphism corresponding to $\F_q$. For $w \in S_n$, $\s \in G$ regular semisimple, $\u \in G$ regular unipotent, we have
\begin{gather*}
[\cl{X(w)}] = \sum_{\substack{u, v \in S_n \\ uw\act{F}v^{-1}=w_0\\\ell(u)+\ell(v)=\ell(w_0)-\ell(w)}}\fS_{u}(x)\fS_{v}(x)q^{\ell(v)}
\\ [\cl{\Y_{w,\s}}] = \sum_{\substack{u, v \in S_n \\ uwv^{-1}=w_0\\\ell(u)+\ell(v)=\ell(w_0)-\ell(w)}}\fS_{u}(x)\fS_{v}(x)
\\ [\cl{\Y_{w,\u}}] = \frac{|(S_n)_I|}{n!}\sum_{\substack{u, v \in S_n \\ uwv^{-1}=w_0\\\ell(u)+\ell(v)=\ell(w_0)-\ell(w)}}\fS_{u}(x)\fS_{v}(x)
\end{gather*}
where $I = \supp(w)$ and $(S_n)_I$ is the parabolic subgroup of $S_n$ corresponding to $I$.
\end{thm}

We consider the double Schubert polynomial $\fS_{w_0}(x;y) = \prod_{i+j\leq n, i,j\geq 1} (x_i-y_j)$. It is known \cite[Corollary 2.4.8]{manivel} that
$$\fS_{w_0}(x;-y) = \sum_{w \in S_n} \fS_w(x)\fS_{ww_0}(y).$$
Suppose $w \in S_n$ has a reduced expression $w=s_1\cdots s_r$ where $r=\ell(w)$. We define $\del_w \colonequals \del_{s_1} \cdots \del_{s_r}$ to be the product of divided difference operators. Here
$$\del_{s_i}f(x_1, \cdots, x_n) \colonequals \frac{f(x_1, \cdots, x_n)-f(x_1, \cdots, x_{i+1}, x_i, \cdots, x_n)}{x_i - x_{i+1}}.$$
Note that this is well-defined on $\Z[x_1, \cdots, x_n]/J$ and the definition of $\del_{w}$ does not depend on the choice of the reduced word. We define $\del_w^x$ to be such an operator which only acts on $x_i$'s.

It is known that $\del_{w}^x\fS_{w'}(x) = \fS_{w'w^{-1}}(x)$ if $\ell(w')-\ell(w^{-1}) = \ell(w'w^{-1})$ and 0 otherwise. Thus
\begin{align*}
\del_{w}^x\fS_{w_0}(x;-y) &= \sum_{\substack{w' \in S_n \\\ell(w')-\ell(w^{-1})=\ell(w'w^{-1})}}\fS_{w'w^{-1}}(x)\fS_{w'w_0}(y).
\end{align*}
Also there exists an involution $\omega_y: \Z[y_1, \cdots, y_n]/J\rightarrow  \Z[y_1, \cdots, y_n]/J : y_i \mapsto -y_{n-i+1}$ which sends $\fS_{w}$ to $\fS_{w_0ww_0}$. Thus
\begin{align*}
\omega_y\del_{w}^x\fS_{w_0}(x;-y) &= \sum_{\substack{w' \in S_n \\\ell(w')-\ell(w^{-1})=\ell(w'w^{-1})}}\fS_{w'w^{-1}}(x)\fS_{w_0w'}(y)
\\&= \sum_{\substack{w' \in S_n \\\ell(w')+\ell(w^{-1})=\ell(w'w^{-1})}}\fS_{w_0w'w^{-1}}(x)\fS_{w'}(y)
\\&= \sum_{\substack{u,v \in S_n\\uwv^{-1} = w_0 \\ \ell(u)+\ell(v)=\ell(w_0)-\ell(w)}}\fS_{u}(x)\fS_{v}(y)
\end{align*}
and furthermore (recall that $\fS_w(x)$ is homogeneous of degree $\ell(w)$)
$$\omega_y\del_{w}^x\fS_{w_0}(x;-qy) = \sum_{\substack{u,v \in S_n\\uwv^{-1} = w_0 \\ \ell(u)+\ell(v)=\ell(w_0)-\ell(w)}}\fS_{u}(x)\fS_{v}(y)q^{\ell(v)}.$$
If we set $y_i = x_i$ for $1\leq i \leq n$, we have the following.
\begin{prop} Let $G$ be of type A$_{n-1}$ and identify $W=S_n$. If $G$ is defined over $\F_q$ and $F$ acts trivially on $S_n$, then 
$$[\cl{\X(w)}]= (\omega_y\del_{w}^x\fS_{w_0}(x;-qy))_{y_i = x_i}.$$
If $F$ acts as a conjugation by $w_0$ on $S_n$, then 
$$[\cl{\X(w)}]= (\del_{w}^x\fS_{w_0}(x;-qy))_{y_i = x_i}.$$
For $\s \in G$ regular semisimple and $\u \in G$ regular unipotent, we have 
\begin{gather*}
[\cl{\Y_{w,\s}}] = (\omega_y\del_{w}^x\fS_{w_0}(x;-y))_{y_i = x_i},
\\ [\cl{\Y_{w,\u}}] = \frac{|(S_n)_I|}{n!} (\omega_y\del_{w}^x\fS_{w_0}(x;-y))_{y_i = x_i}
\end{gather*}
where $I=\supp(w)$ and $(S_n)_I$ is the parabolic subgroup of $S_n$ corresponding to $I$.
\end{prop}
Thus we have a simple algorithm to calculate the classes of such varieties.

\begin{rmk} Proposition \ref{ss:winv} says $[\cl{\Y_{w,\s}}] = [\cl{\Y_{w^{-1}, \s}}]$ for $w\in S_n$; it is trivial to check since the formula has a symmetry. However, it is not combinatorially obvious that $[\cl{\Y_{ww', \s}}] = [\cl{\Y_{w'w, \s}}]$ for $w, w'\in S_n$ such that $\supp(w) \cap \supp(w') = \emptyset$ (Proposition \ref{ss:ww'}.) It would be interesting to find a purely combinatorial proof of this fact.
\end{rmk}

We will give some examples of $[\cl{X(w)}]$ for small groups. 
\begin{example}[Type A${_1}$] For $G=GL_2$ and $W=S_2$, $F$ always acts trivially on $S_2$, and we have
\begin{align*}
&[\cl{X(id)}] = (q+1)[\cl{C_{id}}], &&[\cl{X(s_1)}] = [\cl{C_{s_1}}].
\end{align*}
Note that they are linearly independent if $q+1 \neq 0$ and $(q+1)$ divides $|G^F|_{p'}$. This is expected from Theorem \ref{dl:realbasis}.
\end{example}

\begin{example}[Type A${_2}$] For $G=GL_3$ and $W=S_3$, if $F$ acts trivially on $S_3$, then
\begin{align*}
&[\cl{X(id)}] = (q^3 + 2q^2 + 2q + 1)[\cl{C_{id}}], &&[\cl{X(s_1)}] = (q^2 + q + 1)[\cl{C_{s_1}}], 
\\ &[\cl{X(s_2)}] = (q^2 + q + 1)[\cl{C_{s_2}}], &&[\cl{X(s_1s_2)}] = q[\cl{C_{s_1s_2}}]+[\cl{C_{s_2s_1}}], 
\\ &[\cl{X(s_2s_1)}] = [\cl{C_{s_1s_2}}]+q[\cl{C_{s_2s_1}}], && [\cl{X(s_1s_2s_1)}] = [\cl{C_{s_1s_2s_1}}].
\end{align*}
In the other case, we have
\begin{align*}
&[\cl{X(id)}] = (q^3 + 1)[\cl{C_{id}}], &&[\cl{X(s_1)}] = (q + 1)[\cl{C_{s_1}}]+(q^2 + q)[\cl{C_{s_2}}], 
\\ &[\cl{X(s_2)}] = (q^2 + q)[\cl{C_{s_1}}]+(q+1)[\cl{C_{s_2}}], &&[\cl{X(s_1s_2)}] = (q+1)[\cl{C_{s_2s_1}}], 
\\ &[\cl{X(s_2s_1)}] = (q+1)[\cl{C_{s_1s_2}}], && [\cl{X(s_1s_2s_1)}] = [\cl{C_{s_1s_2s_1}}].
\end{align*}

If we substitute $q$ with 1 on the former formulae, we see that $[\cl{Y_{s_1s_2, \s}}] = [\cl{Y_{s_2s_1, \s}}]$ for regular semisimple $\s\in G$, which is expected from Proposition \ref{ss:winv} and \ref{ss:ww'}. Also note that $[\cl{X(w)}]_{w\in S_3}$ form a basis of $A_*(\B)$ provided $(q^2+q+1)(q+1)(q-1)\neq 0$ in the first case and $(q^2-q+1)(q+1)(q-1)\neq 0$ in the second case. This is also expected from Theorem \ref{dl:realbasis} since $|G^F|_{p'}$ has factors $q^2+q+1, q+1, q-1$ in the first case and $q^2-q+1, q+1, q-1$ in the second case.
\end{example}
From now on we only check when $[\cl{\Y_{w,\s}}]$ is equal for different $w \in W$.

\begin{example}[Type A${_3}$] For $G=GL_4$ and $W=S_4$, the list below shows when $[\cl{\Y_{w,\s}}]$ coincides for different $w\in S_4$.
\begin{gather*}
\{s_1s_2, s_2s_1\}, \{s_2s_3, s_3s_2\}, \{s_3s_1s_2, s_2s_3s_1, s_3s_2s_1, s_1s_2s_3\},
\\\{s_1s_2s_3s_1, s_3s_1s_2s_1\},\{s_1s_2s_3s_2, s_2s_3s_2s_1\},\{s_2s_3s_1s_2s_1, s_1s_2s_3s_1s_2\}
\end{gather*}
They are also expected from Proposition \ref{ss:winv} and \ref{ss:ww'}.
\end{example}

\begin{example}[Type A${_4}$] For $G=GL_5$ and $W=S_5$, the list below shows when $[\cl{\Y_{w,\s}}]$ coincides for different $w\in S_5$.
\begin{gather*}
\{s_2s_1,s_1s_2\},\{s_3s_2,s_2s_3\},\{s_3s_4,s_4s_3\},
\{s_3s_4s_1,s_4s_3s_1\},\{s_4s_2s_1,s_4s_1s_2\},
\\\{s_3s_2s_1,s_2s_3s_1,s_3s_1s_2,s_1s_2s_3\},\{s_3s_4s_2,s_2s_3s_4,s_4s_3s_2,s_4s_2s_3\},
\\\{s_2s_3s_4s_2,s_4s_2s_3s_2\},\{s_3s_1s_2s_1,s_1s_2s_3s_1\},\{s_2s_3s_2s_1,s_1s_2s_3s_2\},\{s_3s_4s_3s_2,s_2s_3s_4s_3\},
\\\{s_3s_4s_2s_1,s_2s_3s_4s_1,s_4s_3s_2s_1,s_4s_2s_3s_1,s_3s_4s_1s_2,s_1s_2s_3s_4,s_4s_3s_1s_2,s_4s_1s_2s_3\},
\\\{s_2s_3s_1s_2s_1,s_1s_2s_3s_1s_2\},\{s_3s_4s_2s_3s_1,s_3s_4s_1s_2s_3\},\{s_2s_3s_4s_1s_2,s_4s_2s_3s_1s_2\},
\\\{s_2s_3s_4s_2s_3,s_3s_4s_2s_3s_2\},\{s_3s_4s_2s_3s_1s_2,s_2s_3s_4s_1s_2s_3\},
\\\{s_3s_4s_1s_2s_1,s_1s_2s_3s_4s_1,s_4s_3s_1s_2s_1,s_4s_1s_2s_3s_1\},\{s_2s_3s_4s_2s_1,s_4s_2s_3s_2s_1,s_1s_2s_3s_4s_2,s_4s_1s_2s_3s_2\},
\\\{s_3s_4s_3s_2s_1,s_2s_3s_4s_3s_1,s_3s_4s_3s_1s_2,s_1s_2s_3s_4s_3\},\{s_1s_2s_3s_4s_2s_1,s_4s_1s_2s_3s_2s_1\},
\\\{s_2s_3s_4s_3s_2s_1,s_1s_2s_3s_4s_3s_2\},\{s_3s_4s_3s_1s_2s_1,s_1s_2s_3s_4s_3s_1\},
\\\{s_2s_3s_4s_2s_3s_1,s_3s_4s_2s_3s_2s_1,s_1s_2s_3s_4s_2s_3,s_3s_4s_1s_2s_3s_2\},
\\\{s_2s_3s_4s_1s_2s_1,s_4s_2s_3s_1s_2s_1,s_1s_2s_3s_4s_1s_2,s_4s_1s_2s_3s_1s_2\},
\\\{s_2s_3s_4s_2s_3s_1s_2,s_2s_3s_4s_1s_2s_3s_2\},\{s_2s_3s_4s_3s_1s_2s_1,s_1s_2s_3s_4s_3s_1s_2\},
\\\{s_1s_2s_3s_4s_1s_2s_1,s_4s_1s_2s_3s_1s_2s_1\},\{s_1s_2s_3s_4s_2s_3s_1,s_3s_4s_1s_2s_3s_2s_1\},
\\\{s_2s_3s_4s_1s_2s_3s_1,s_3s_4s_1s_2s_3s_1s_2\},\{s_3s_4s_2s_3s_1s_2s_1,s_1s_2s_3s_4s_1s_2s_3\},
\\\{s_2s_3s_4s_2s_3s_2s_1,s_1s_2s_3s_4s_2s_3s_2\},\{s_2s_3s_4s_2s_3s_1s_2s_1,s_1s_2s_3s_4s_1s_2s_3s_2\},
\\\{s_2s_3s_4s_1s_2s_3s_2s_1,s_1s_2s_3s_4s_2s_3s_1s_2\},\{s_1s_2s_3s_4s_1s_2s_3s_1,s_3s_4s_1s_2s_3s_1s_2s_1\},
\\\{s_2s_3s_4s_1s_2s_3s_1s_2s_1,s_1s_2s_3s_4s_1s_2s_3s_1s_2\},\{s_1s_2s_3s_4s_2s_3s_1s_2s_1,s_1s_2s_3s_4s_1s_2s_3s_2s_1\}.
\end{gather*}
They are also expected from Proposition \ref{ss:winv} and \ref{ss:ww'}.
\end{example}

\begin{example}[Type A${_5}$] For $G=GL_6$ and $W=S_6$, there are six exceptions not explained by Proposition \ref{ss:winv} or \ref{ss:ww'} listed below.
\begin{gather*}
\{s_1s_2s_3s_4s_5s_3s_4s_1s_2, s_2s_3s_4s_5s_3s_4s_1s_2s_1, s_4s_5s_2s_3s_4s_3s_1s_2s_1, s_4s_5s_1s_2s_3s_4s_3s_1s_2\},
\\ \{s_2s_3s_4s_5s_2s_3s_4s_1s_2s_1, s_1s_2s_3s_4s_5s_1s_2s_3s_4s_2, 
s_4s_5s_1s_2s_3s_4s_1s_2s_3s_2, s_4s_5s_2s_3s_4s_2s_3s_1s_2s_1\},
\\\{s_2s_3s_4s_5s_1s_2s_3s_4s_1s_2s_3, s_2s_3s_4s_5s_2s_3s_4s_1s_2s_3s_1, s_3s_4s_5s_2s_3s_4s_1s_2s_3s_1s_2, s_3s_4s_5s_1s_2s_3s_4s_1s_2s_3s_2\},
\\\{s_1s_2s_3s_4s_5s_1s_2s_3s_4s_2s_1,s_1s_2s_3s_4s_5s_2s_3s_4s_1s_2s_1, s_4s_5s_1s_2s_3s_4s_1s_2s_3s_2s_1, s_4s_5s_1s_2s_3s_4s_2s_3s_1s_2s_1\},
\\\{s_2s_3s_4s_5s_3s_4s_2s_3s_1s_2s_1, s_2s_3s_4s_5s_2s_3s_4s_3s_1s_2s_1, s_1s_2s_3s_4s_5s_3s_4s_1s_2s_3s_2, s_1s_2s_3s_4s_5s_1s_2s_3s_4s_3s_2\},
\\\{s_1s_2s_3s_4s_5s_3s_4s_1s_2s_3s_2s_1, s_1s_2s_3s_4s_5s_1s_2s_3s_4s_3s_2s_1,
\\ s_1s_2s_3s_4s_5s_3s_4s_2s_3s_1s_2s_1, s_1s_2s_3s_4s_5s_2s_3s_4s_3s_1s_2s_1\}.
\end{gather*}
It would be interesting to find a geometric/combinatorial condition which is both necessary and sufficient to find $w, w' \in W$ such that $[\cl{\Y_{w,\s}}] = [\cl{\Y_{w', \s}}]$ for $\s\in G$ regular semisimple.
\end{example}

\bibliographystyle{alpha}
\bibliography{decomp_revised_1st}

\end{document}